\documentclass[12pt]{amsart}
\usepackage{amssymb,latexsym,tikz}
\usepackage[top=40mm, bottom=25mm, left=35mm, right=35mm]{geometry}
\usepackage{appendix}

\usetikzlibrary{arrows,decorations.markings}
\tikzstyle arrowstyle=[scale=1]
\tikzstyle directed=[postaction={decorate,
decoration={markings,mark=at position .65 with {\arrow[arrowstyle]{stealth}}}}]

\tikzset{mid vert/.style={/utils/exec=\tikzset{every node/.append style={outer sep=0.8ex}},
postaction=decorate,decoration={markings,
mark=at position 0.5 with {\draw[-] (0,#1) -- (0,-#1);}}},
mid vert/.default=0.75ex}

\usepackage{mathtools}
\usepackage{xypic}
\usepackage[utf8]{inputenc}

\begin{document}

\title 
[Surprising order structures in mathematics]
{Surprising occurrences of order structures in mathematics}

\author{Gunnar Fl{\o}ystad}
\address{Matematisk institutt\\
         Postboks 7803\\
         5020 Bergen}
\email{gunnar@mi.uib.no}



\date{\today}



\theoremstyle{plain}
\newtheorem{theorem}{Theorem}[section]
\newtheorem{corollary}[theorem]{Corollary}
\newtheorem*{main}{Main Theorem}
\newtheorem{lemma}[theorem]{Lemma}
\newtheorem{proposition}[theorem]{Proposition}
\newtheorem{conjecture}[theorem]{Conjecture}
\newtheorem{theoremp}{Theorem}

\theoremstyle{definition}
\newtheorem{definition}[theorem]{Definition}
\newtheorem{fact}[theorem]{Fact}
\newtheorem{obs}[theorem]{Observation}
\newtheorem{definisjon}[theorem]{Definisjon}
\newtheorem{problem}[theorem]{Problem}
\newtheorem{condition}[theorem]{Condition}

\theoremstyle{remark}
\newtheorem{notation}[theorem]{Notation}
\newtheorem{remark}[theorem]{Remark}
\newtheorem{example}[theorem]{Example}
\newtheorem{claim}{Claim}
\newtheorem{observation}[theorem]{Observation}
\newtheorem{question}[theorem]{Question}


\newcommand{\psp}[1]{{{\bf P}^{#1}}}
\newcommand{\psr}[1]{{\bf P}(#1)}
\newcommand{\op}{{\mathcal O}}
\newcommand{\opw}{\op_{\psr{W}}}

\newcommand{\ini}[1]{\text{in}(#1)}
\newcommand{\gin}[1]{\text{gin}(#1)}
\newcommand{\kr}{{\Bbbk}}
\newcommand{\pd}{\partial}
\newcommand{\vardel}{\partial}


\newcommand{\coh}{{{\text{{\rm coh}}}}}


\newcommand{\modv}[1]{{#1}\text{-{mod}}}
\newcommand{\modstab}[1]{{#1}-\underline{\text{mod}}}

\newcommand{\sut}{{}^{\tau}}
\newcommand{\sumit}{{}^{-\tau}}
\newcommand{\til}{\thicksim}

\newcommand{\totp}{\text{Tot}^{\prod}}
\newcommand{\dsum}{\bigoplus}
\newcommand{\dprod}{\prod}
\newcommand{\lsum}{\oplus}
\newcommand{\lprod}{\Pi}

\newcommand{\La}{{\Lambda}}

\newcommand{\sirstj}{\circledast}

\newcommand{\she}{\EuScript{S}\text{h}}
\newcommand{\cm}{\EuScript{CM}}
\newcommand{\cmd}{\EuScript{CM}^\dagger}
\newcommand{\cmri}{\EuScript{CM}^\circ}
\newcommand{\cler}{\EuScript{CL}}
\newcommand{\clerd}{\EuScript{CL}^\dagger}
\newcommand{\clerri}{\EuScript{CL}^\circ}
\newcommand{\gor}{\EuScript{G}}
\newcommand{\cF}{\mathcal{F}}
\newcommand{\cG}{\mathcal{G}}
\newcommand{\cM}{\mathcal{M}}
\newcommand{\cE}{\mathcal{E}}
\newcommand{\cI}{\mathcal{I}}
\newcommand{\cP}{\mathcal{P}}
\newcommand{\cK}{\mathcal{K}}
\newcommand{\cS}{\mathcal{S}}
\newcommand{\cC}{\mathcal{C}}
\renewcommand{\cD}{\mathcal{D}}
\newcommand{\cO}{\mathcal{O}}
\newcommand{\cJ}{\mathcal{J}}
\newcommand{\cU}{\mathcal{U}}
\newcommand{\cQ}{\mathcal{Q}}
\newcommand{\cX}{\mathcal{X}}
\newcommand{\cY}{\mathcal{Y}}
\newcommand{\cZ}{\mathcal{Z}}
\newcommand{\cV}{\mathcal{V}}

\newcommand{\mm}{\mathfrak{m}}

\newcommand{\dlim} {\varinjlim}
\newcommand{\ilim} {\varprojlim}

\newcommand{\CM}{\text{CM}}
\newcommand{\Mon}{\text{Mon}}


\newcommand{\Kom}{\text{Kom}}


\newcommand{\EH}{{\mathbf H}}
\newcommand{\res}{\text{res}}
\newcommand{\Hom}{\text{Hom}}
\newcommand{\inhom}{{\underline{\text{Hom}}}}
\newcommand{\Ext}{\text{Ext}}
\newcommand{\Tor}{\text{Tor}}
\newcommand{\ghom}{\mathcal{H}om}
\newcommand{\gext}{\mathcal{E}xt}
\newcommand{\id}{\text{{id}}}
\newcommand{\im}{\text{im}\,}
\newcommand{\codim} {\text{codim}\,}
\newcommand{\resol}{\text{resol}\,}
\newcommand{\rank}{\text{rank}\,}
\newcommand{\lpd}{\text{lpd}\,}
\newcommand{\coker}{\text{coker}\,}
\newcommand{\supp}{\text{supp}\,}
\newcommand{\Ad}{A_\cdot}
\newcommand{\Bd}{B_\cdot}
\newcommand{\Fd}{F_\cdot}
\newcommand{\Gd}{G_\cdot}


\newcommand{\sus}{\subseteq}
\newcommand{\sups}{\supseteq}
\newcommand{\pil}{\rightarrow}
\newcommand{\vpil}{\leftarrow}
\newcommand{\rpil}{\leftarrow}
\newcommand{\lpil}{\longrightarrow}
\newcommand{\inpil}{\hookrightarrow}
\newcommand{\pils}{\twoheadrightarrow}
\newcommand{\projpil}{\dashrightarrow}
\newcommand{\dotpil}{\dashrightarrow}
\newcommand{\adj}[2]{\overset{#1}{\underset{#2}{\rightleftarrows}}}
\newcommand{\mto}[1]{\stackrel{#1}\longrightarrow}
\newcommand{\vmto}[1]{\stackrel{#1}\longleftarrow}
\newcommand{\mtoelm}[1]{\stackrel{#1}\mapsto}
\newcommand{\bihom}[2]{\overset{#1}{\underset{#2}{\rightleftarrows}}}
\newcommand{\eqv}{\Leftrightarrow}
\newcommand{\impl}{\Rightarrow}

\newcommand{\iso}{\cong}
\newcommand{\te}{\otimes}
\newcommand{\into}[1]{\hookrightarrow{#1}}
\newcommand{\ekv}{\Leftrightarrow}
\newcommand{\equi}{\simeq}
\newcommand{\isopil}{\overset{\cong}{\lpil}}
\newcommand{\equipil}{\overset{\equi}{\lpil}}
\newcommand{\ispil}{\isopil}
\newcommand{\vvi}{\langle}
\newcommand{\hvi}{\rangle}
\newcommand{\susneq}{\subsetneq}
\newcommand{\sgn}{\text{sign}}


\newcommand{\xd}{\check{x}}
\newcommand{\ortog}{\bot}
\newcommand{\tL}{\tilde{L}}
\newcommand{\tM}{\tilde{M}}
\newcommand{\tH}{\tilde{H}}
\newcommand{\tvH}{\widetilde{H}}
\newcommand{\tvh}{\widetilde{h}}
\newcommand{\tV}{\tilde{V}}
\newcommand{\tS}{\tilde{S}}
\newcommand{\tT}{\tilde{T}}
\newcommand{\tR}{\tilde{R}}
\newcommand{\tf}{\tilde{f}}
\newcommand{\ts}{\tilde{s}}
\newcommand{\tp}{\tilde{p}}
\newcommand{\tr}{\tilde{r}}
\newcommand{\tfst}{\tilde{f}_*}
\newcommand{\empt}{\emptyset}
\newcommand{\bfa}{{\mathbf a}}
\newcommand{\bfb}{{\mathbf b}}
\newcommand{\bfd}{{\mathbf d}}
\newcommand{\bfl}{{\mathbf \ell}}
\newcommand{\bfx}{{\mathbf x}}
\newcommand{\bfm}{{\mathbf m}}
\newcommand{\bfv}{{\mathbf v}}
\newcommand{\bft}{{\mathbf t}}
\newcommand{\bbfa}{{\mathbf a}^\prime}
\newcommand{\la}{\lambda}
\newcommand{\bfen}{{\mathbf 1}}
\newcommand{\bfe}{{\mathbf 1}}
\newcommand{\ep}{\epsilon}
\newcommand{\en}{r}
\newcommand{\tu}{s}
\newcommand{\Sym}{\text{Sym}}

\newcommand{\ome}{\omega_E}

\newcommand{\bevis}{{\bf Proof. }}
\newcommand{\demofin}{\qed \vskip 3.5mm}
\newcommand{\nyp}[1]{\noindent {\bf (#1)}}
\newcommand{\demo}{{\it Proof. }}
\newcommand{\demodone}{\demofin}
\newcommand{\parg}{{\vskip 2mm \addtocounter{theorem}{1}  
                   \noindent {\bf \thetheorem .} \hskip 1.5mm }}

\newcommand{\lcm}{{\text{lcm}}}


\newcommand{\dl}{\Delta}
\newcommand{\cdel}{{C\Delta}}
\newcommand{\cdelp}{{C\Delta^{\prime}}}
\newcommand{\dlst}{\Delta^*}
\newcommand{\Sdl}{{\mathcal S}_{\dl}}
\newcommand{\lk}{\text{lk}}
\newcommand{\lkd}{\lk_\Delta}
\newcommand{\lkp}[2]{\lk_{#1} {#2}}
\newcommand{\del}{\Delta}
\newcommand{\delr}{\Delta_{-R}}
\newcommand{\dd}{{\dim \del}}
\newcommand{\Del}{\Delta}

\renewcommand{\aa}{{\bf a}}
\newcommand{\bb}{{\bf b}}
\newcommand{\cc}{{\bf c}}
\newcommand{\xx}{{\bf x}}
\newcommand{\yy}{{\bf y}}
\newcommand{\zz}{{\bf z}}
\newcommand{\mv}{{\xx^{\aa_v}}}
\newcommand{\mF}{{\xx^{\aa_F}}}

\newcommand{\Symm}{\text{Sym}}
\newcommand{\pnm}{{\bf P}^{n-1}}
\newcommand{\opnm}{{\go_{\pnm}}}
\newcommand{\ompnm}{\omega_{\pnm}}

\newcommand{\pn}{{\bf P}^n}
\newcommand{\hele}{{\mathbb Z}}
\newcommand{\nat}{{\mathbb N}}
\newcommand{\rasj}{{\mathbb Q}}
\newcommand{\bfone}{{\mathbf 1}}

\newcommand{\dt}{\bullet}
\newcommand{\disk}{\scriptscriptstyle{\bullet}}

\newcommand{\cxF}{F_\dt}
\newcommand{\pol}{f}

\newcommand{\Rn}{{\mathbb R}^n}
\newcommand{\An}{{\mathbb A}^n}
\newcommand{\frg}{\mathfrak{g}}
\newcommand{\PW}{{\mathbb P}(W)}

\newcommand{\pos}{{\mathcal Pos}}
\newcommand{\g}{{\gamma}}

\newcommand{\Vaa}{V_0}
\newcommand{\Bp}{B^\prime}
\newcommand{\Bpp}{B^{\prime \prime}}
\newcommand{\bbp}{\mathbf{b}^\prime}
\newcommand{\bbpp}{\mathbf{b}^{\prime \prime}}
\newcommand{\bp}{{b}^\prime}
\newcommand{\bpp}{{b}^{\prime \prime}}

\newcommand{\oLa}{\overline{\Lambda}}
\newcommand{\ov}[1]{\overline{#1}}
\newcommand{\ovv}[1]{\overline{\overline{#1}}}
\newcommand{\tm}{\tilde{m}}
\newcommand{\po}{\bullet}

\newcommand{\surj}[1]{\overset{#1}{\twoheadrightarrow}}
\newcommand{\Supp}{\text{Supp}}

\def\CC{{\mathbb C}}
\def\GG{{\mathbb G}}
\def\ZZ{{\mathbb Z}}
\def\NN{{\mathbb N}}
\def\RR{{\mathbb R}}
\def\OO{{\mathbb O}}
\def\QQ{{\mathbb Q}}
\def\VV{{\mathbb V}}
\def\PP{{\mathbb P}}
\def\EE{{\mathbb E}}
\def\FF{{\mathbb F}}
\def\AA{{\mathbb A}}

\newcommand{\oR}{\overline{R}}
\newcommand{\bfu}{{\mathbf u}}
\newcommand{\nn}{{\mathbf n}}
\newcommand{\oa}{\overline{a}}
\newcommand{\cop}{\text{cop}}
\renewcommand{\op}{\text{op}\,}
\renewcommand{\mm}{{\mathbf m}}
\newcommand{\ngmi}{\text{neg}}
\newcommand{\up}{\text{up}}
\newcommand{\dw}{\text{down}}
\newcommand{\diw}[1]{\widehat{#1}}
\newcommand{\di}[1]{\diw{#1}}
\newcommand{\bo}{b}
\newcommand{\ub}{u}
\newcommand{\fs}{\infty}
\newcommand{\ifst}{\infty}
\newcommand{\mon}{{mon}}
\newcommand{\cl}{\text{cl}}
\newcommand{\intr}{\text{int}}
\newcommand{\ul}[1]{\underline{#1}}
\renewcommand{\ov}[1]{\overline{#1}}
\newcommand{\bipil}{\leftrightarrow}
\newcommand{\bfc}{{\mathbf c}}
\renewcommand{\mp}{m^\prime}
\newcommand{\np}{n^\prime}
\newcommand{\Mod}{\text{Mod }}
\newcommand{\Sh}{\text{Sh } }
\newcommand{\st}{\text{st}}
\newcommand{\hM}{\tilde{M}}
\newcommand{\hs}{\tilde{s}}
\newcommand{\ee}{\mathbf{e}}
\renewcommand{\dd}{\mathbf{d}}
\renewcommand{\en}{{\mathbf 1}}
\long\def\ignore#1{}
\newcommand{\lex}{{\text{lex}}}
\newcommand{\ordGL}{\succeq_{\lex}}
\newcommand{\ordG}{\succ_{\lex}}
\newcommand{\ordML}{\preceq_{\lex}}
\newcommand{\ordM}{\prec_{\lex}}
\newcommand{\tLa}{\tilde{\Lambda}}
\newcommand{\tGa}{\tilde{\Gamma}}
\newcommand{\STS}{\text{STS}}
\newcommand{\ii}{{\rotatebox[origin=c]{180}{\scalebox{0.7}{\rm{!}}}}}
\newcommand{\jj}{\mathbf{j}}
\renewcommand{\mod}{\text{ mod}\,}
\newcommand{\shmod}{\texttt{shmod}\,}
\newcommand{\hf}{\underline{f}}
\newcommand{\Glim}{\lim}
\newcommand{\Gcolim}{\colim}
\newcommand{\fm}{f^{\underline{m}}}
\newcommand{\fn}{f^{\underline{n}}}
\newcommand{\gn}{g_{\mathbf{|}n}}
\newcommand{\se}[1]{\overline{#1}}
\newcommand{\cB}{{\mathcal{B}}}
\newcommand{\fin}{{\text{fin}}}
\newcommand{\Poset}{{\text{\bf Poset}}}
\newcommand{\Set}{{\text{\bf Set}}}
\newcommand{\Homi}{\text{Hom}^{L}}
\newcommand{\Homb}{\text{Hom}_S}
\newcommand{\Homu}{\text{Hom}^u}
\newcommand{\bfnu}{{\mathbf 0}}
\renewcommand{\bfen}{{\mathbf 1}}
\newcommand{\semip}{up semi-finite }
\newcommand{\semim}{down semi-finite }
\newcommand{\dif}[1]{\di{#1}_{fin}}
\newcommand{\llin}{\raisebox{1pt}{\scalebox{1}[0.6]{$\mid$}}}
\newcommand{\promap}{\mathrlap{{\hskip 2.8mm}{\llin}}{\lpil}}
\newcommand{\dual}{{\widehat{}}}
\newcommand{\Bool}{{\bf Bool }}
\newcommand{\pro}{{pro}}
\newcommand{\ovF}{\overline{F}}
\newcommand{\TO}{{\mathrm to}}
\newcommand{\dT}{\overset{\rightarrow}{T}}
\def\cl{\overline}

\newlength{\dhatheight}
\newcommand{\doublehat}[1]{%
    \settoheight{\dhatheight}{\ensuremath{\widehat{#1}}}%
    \addtolength{\dhatheight}{-0.35ex}%
    \widehat{\vphantom{\rule{1pt}{\dhatheight}}%
    \smash{\widehat{#1}}}}

\newcommand{\napo}{natural }

\newcommand{\colim}{\text{colim}}
\newcommand{\oPsi}{\overline{\Psi}}
\newcommand{\bfC}{{\mathbf C}}
\newcommand{\ic}{{\mathfrak c}}
\newcommand{\iC}{{\mathfrak C}}
\newcommand{\set}{{\bf set}}

\newcommand{\ux}{U_x}
\newcommand{\uy}{U_y}
\newcommand{\glv}{GL(V)}
\newcommand{\gln}{GL(n)}
\newcommand{\glx}{GL(X)}
\newcommand{\glxo}[1]{GL_{\geq_{#1}}(X)}
\newcommand{\cT}{{\mathcal T}}
\newcommand{\ben}{{1}}
\newcommand{\oX}{\overline{X}}
\newcommand{\Pre}{\text{Pre}}
\newcommand{\pre}{\text{pre}}
\newcommand{\kQ}{k\langle Q \rangle}
\newcommand{\preo}{{\tt{pre}}}
\newcommand{\Preo}{{\tt{Preorder}}}
\newcommand{\Top}{\text{Top}}
\newcommand{\topt}{{\tt{t}}}
\newcommand{\Digraph}{{\tt{Digraph}}}
\newcommand{\smCat}{{\tt{smallCat}}}
\newcommand{\dig}{{\tt{dig}}}
\newcommand{\path}{{\tt{path}}}
\newcommand{\inc}{{\tt{inc}}}
\newcommand{\bpreceq}{{\mathbf {\preceq}}}
\newcommand{\bsucceq}{{\mathbf {\succeq}}}

\begin{abstract}
Order and symmetry are main structural principles in
mathematics. We give five examples where on the face of it order is
not apparent, but deeper investigations reveal that they are governed
by order structures. These examples are finite topologies, associative
algebras, subgroups of matrix groups, ideals in polynomial rings,
and classes of bipartite graphs.
\end{abstract}

\maketitle


\newcommand{\topre}{
 \begin{tikzpicture}[scale=0.7, vertices/.style={draw, fill=black, circle, inner sep=1.5pt}]
\draw [help lines, white] (-1,0) grid (1,1);

\node (1x) at (0-2,0) {x.};
\node [vertices] (1) at (0-1.4,0) {};
\node [vertices] (2) at (0.4-1.4,0) {};


\node (a) at (0.4,0) {y.};
\node [vertices] (a) at (1,0) {};
\node [vertices] (b) at (1,0.7) {};

\foreach \to/\from in {a/b}
\draw (\to)--(\from);


\node (ahop) at (2.8, 0) {z.};
\node (aho) at (2+1.5,0.2) {};
\draw (aho) circle (0.3);
\filldraw[gray,thin] (1.9+1.5,0.1) circle (0.1);
\filldraw[gray,thin] (2.1+1.5,0.3) circle (0.1);


\end{tikzpicture}
}


\newcommand{\trepre}{
 \begin{tikzpicture}[scale=0.7, vertices/.style={draw, fill=black, circle, inner sep=1.5pt}]
\draw [help lines, white] (-1,0) grid (1,1);

\node (11a) at (-0.6,0) {a.};
\node [vertices] (11) at (0,0) {};
\node [vertices] (12) at (0.4,0) {};
\node [vertices] (13) at (0.8,0) {};

\node (21b) at (2.8-0.8,0) {b.};
\node [vertices] (21) at (2.8-0.2,0) {};
\node [vertices] (22) at (2.8-0.2,0.6) {};
\node [vertices] (23) at (3.2-0.2,0) {};

\draw (21)--(22);

\node (31c) at (5.2-1,0) {c.};
\node [vertices] (31) at (5.2-0.4,0) {};
\node [vertices] (32) at (5.0-0.4,0.5) {};
\node [vertices] (33) at (5.4-0.4,0.5) {};

\draw (31)--(32);
\draw (31)--(33);

+1.8
\node (31c) at (5.2-1+1.8,0) {d.};
\node [vertices] (31) at (5.2-0.4+1.8,0.5) {};
\node [vertices] (32) at (5.0-0.4+1.8,0) {};
\node [vertices] (33) at (5.4-0.4+1.8,0) {};

\draw (31)--(32);
\draw (31)--(33);

\node (41d) at (6+1.8,0) {e.};
\node [vertices] (41) at (6.8+1.8,0) {};
\node [vertices] (42) at (6.8+1.8,0.5) {};
\node [vertices] (43) at (6.8+1.8,1) {};

\draw (41)--(42);
\draw (42)--(43);


\node (ahoe) at (7.8+1.8,0.0) {f.};
\node (aho) at (8.6+1.8,0.1) {};
\draw (aho) circle (0.3);
\filldraw[gray,thin] (8.5+1.8,0) circle (0.1);
\filldraw[gray,thin] (8.7+1.8,0.2) circle (0.1);
\node [vertices] (53) at (9.2+1.8,0) {};

\node (ahof) at (10.2+1.8,0.0) {g.};
\node (aho) at (11+1.8,0.1) {};
\node (6v) at (11+1.8,0.2) {};
\draw (aho) circle (0.3);
\filldraw[gray,thin] (10.9+1.8,0) circle (0.1);
\filldraw[gray,thin] (11.1+1.8,0.2) circle (0.1);
\node [vertices] (63) at (11+1.8,0.9) {};
\draw (6v)--(63);

\node (ahog) at (12+1.8,0.0) {h.};
\node (aho) at (12.8+1.8,0.9) {};
\node (7v) at (12.8+1.8,0.8) {};
\draw (aho) circle (0.3);
\filldraw[gray,thin] (12.7+1.8,0.8) circle (0.1);
\filldraw[gray,thin] (12.9+1.8,1) circle (0.1);
\node [vertices] (71) at (12.8+1.8,0) {};
\draw (71)--(7v);

\node (ahoh) at (13.8+1.8,0.0) {i.};
\node (aho) at (14.6+1.8,0.2) {};
\draw (aho) circle (0.3);
\filldraw[gray,thin] (14.6+1.8,0.3) circle (0.1);
\filldraw[gray,thin] (14.5+1.8,0.1) circle (0.1);
\filldraw[gray,thin] (14.7+1.8,0.1) circle (0.1);

\end{tikzpicture}
}

\newcommand{\digraph}{
 \begin{tikzpicture}[scale=1, vertices/.style={draw, fill=black, circle, inner sep=1.5pt}]
   \draw [help lines, white] (-1,0) grid (1,1);
   \node (a) at (0,0) {$a$};
   \node (b) at (1.5,1) {$b$};
   \node (c) at (3,0) {$c$};

   \draw (a) circle (0.17);
   \draw (b) circle (0.17);
   \draw (c) circle (0.17);

   \path[->] (a) edge [bend left] node [above] {e} (b);
   \path[->] (a) edge [bend right] node [right] {f} (b);
   \path[->] (b) edge node [above] {g} (c);
   \path[->] (a) edge node [below] {h} (c);
 \end{tikzpicture}
   }

   \newcommand{\bub}[1]{
 \begin{tikzpicture}[scale=0.7, vertices/.style={draw, fill=black, circle, inner sep=1.5pt}]
     \node (b) at (0,0) {#1};
     \draw (b) circle (0.5);
   \end{tikzpicture}
 }

    \newcommand{\linegraph}{
      \begin{tikzpicture}[scale=1, vertices/.style={draw, fill=black, circle, inner sep=1.5pt}]
        
  \node (al) at (0,0.3) {$a$};
   \node (bl) at (1,0.3) {$b$};
   \node (cl) at (2,0.3) {$c$};
   \node (dl) at (3,0.3) {$d$};
        
   \node [vertices] (a) at (0,0) {};
   \node  [vertices] (b) at (1,0) {};
   \node  [vertices] (c) at (2,0) {};
   \node  [vertices] (d) at (3,0) {};
   

\draw (a)--(b)--(c)--(d);   
   \end{tikzpicture}
 }

    \newcommand{\squaregraph}{
      \begin{tikzpicture}[scale=0.7, vertices/.style={draw, fill=black, circle, inner sep=1.5pt}]

 \node (al) at (-0.4,1.2) {$a$};
   \node (bl) at (1.3,1.2) {$b$};
   \node (cl) at (1.3,0.0) {$c$};
   \node (dl) at (-0.4,0.0) {$d$};        
       
   \node [vertices](a) at (0,1) {};
   \node [vertices](b) at (1,1) {};
   \node [vertices](c) at (1,0) {};
   \node [vertices](d) at (0,0) {};
   

\draw (a)--(b)--(c)--(d)--(a);   
   \end{tikzpicture}
 }

\setcounter{tocdepth}{1} 
\tableofcontents

\section{Introduction}

Proving a mathematical statement is to give an argument assuming
simpler mathematical statements. These statements are again proven from
even simpler ones. So it goes on, until reaching the basis, the axioms.

Classifying and unravelling mathematical structure should be to understand
it in terms of simpler structures. For instance, how it is
built up or understood in terms of groups, partial orders, graphs and so on.
Of course, these structures are also complicated, and understanding and
classifying subclasses of such provides in itself deep research into
structure.

Given two objects $a$ and $b$, the simplest way they can relate is
that one "precedes" the other, or that there is a ``direction'' from one
to the other, say $ab$ or the other way $ba$.
They could also be unrelated, or we could allow both $ab$ and $ba$ to
occur. This most basic of ways to relate should be 
fundamental. If $V$ is a set of objects,
the simple mathematical notion for it is a relation $R \sus V \times V$.
It may also be reasonable to have the trivial $a$ precedes $a$, and
that the notion composes:
Given $ab$ and $bc$ one has $ac$.
The relation $R$ is then a {\it preorder}.

One may enrich and allow $ab$, that is $a$ precedes $b$, to occur in several
ways. This gives a set of such, $\Hom(a,b)$, for each ordered pair $(a,b)$.
Equivalently this is a map $E \pil V \times V$ with $\Hom(a,b)$ the
inverse image of $(a,b)$. This structure is a directed graph.
Again, if one would like to compose ordered ways, this gives a
{\it category}. 



We give five examples of well-known and richly studied structures which
on the face of it do not involve order structures, but deeper investigations
reveal that order structures classify or describe essential features
of them.

\medskip
\noindent{\sc Topologies.} Consider topologies on a finite set $X$.
For topologists is is well-known that giving such a topology is equivalent
to give a preorder on $X$, but this may not be so well-known outside that
community.

\medskip
\noindent{\sc Associative algebras.} Associative algebras make a large class.
For a field $k$, the {\it free} $k$-algebra on a set $X$ has basis
all words on the set $X$. A word is a sequence of elements of $X$ and
as such ordered. Going to the opposite in size, we consider the class
of associative $k$-algebras which are finite-dimensional as $k$-vector spaces.
Essential in classifying them are directed graphs, or quivers as they
are called in that community. The associated category and path algebra,
together with relations give essentially all associative $k$-algebras
which are finite-dimensional.

\medskip
\noindent {\sc Subgroups of matrix groups.} The general linear group
$GL_k(n)$ is the group of invertible $n \times n$-matrices with entries
in the field $k$. We consider subgroups of this containing the diagonal
matrices. Surprisingly, and maybe not well-known, connected such
subgroups are precisely classified by preorders on a set with
$n$ elements.

\medskip
\noindent {\sc Ideals in polynomial rings.} These are the classical
objects in algebra and the fundamental algebraic objects in algebraic geometry,
as their zero set gives varieties. A polynomial ring is commutative and
so in contrast to associative algebras, there is no hint of order
relation.

Polynomial ideals may be degenerated to monomial ideals. And monomial
ideals may be degenerated to even simpler types of monomial ideals.
But at the bottom there is a class of monomial ideals which are
as degenerate as possible. Which are these?
To any monomial ideal is associated a preorder. This is a simple
observation which in its explicit form may be new here. And the monomial
ideal is maximally degenerate iff the associate preorder is a total preorder.

\medskip
\noindent {\sc Bipartite graphs.} To a combinatorial object one may often
associate an algebraic object. Thus, one can use the machine and notions
of algebra to investigate combinatorial objects.
To simplicial complexes one may for instance associate Stanley-Reisner
ideals, to graphs one may associate edge ideals in polynomial rings.

A basic niceness condition on commutative rings is that of being
Cohen-Macaulay. For graphs it is as difficult question as anything to
to give a general description of when a graph has Cohen-Macaulay edge ring.
But for bipartite graphs this is completely understood. An amazing
result of J.Herzog and T.Hibi \cite{HeHi-bip} says that this holds
if and only if the bipartite graphs are constructed from a partial order.

\medskip
The organization of this article is as follows. Section \ref{sec:cat}
recalls the basic language of categories and shows that the simplest
of such are monoids, groups, and preorders. Section \ref{sec:digraph}
gives the even more
basic order structure of digraphs and how they induce categories and
preorders. Sections 4, 5, 6, 7, 8 consider respectively topologies, associative
algebras, matrix groups, ideals in polynomial rings, and bipartite graphs.
Appendix A recalls the notions of functors and adjoint functors. 

\medskip
\noindent {\it Acknowledgements.} We thank Hugh Thomas for pointing
out some inacurrancies in the section Associative algebra, in a
previous version. We also thank an anonymous referee for general
suggestions for improving the article.

\section{Categories, groups, and preorders}
\label{sec:cat}

Category theory started with the 1945 article ``A general theory of
natural equivalences'' \cite{EiML} by Samuel Eilenberg and Saunders MacLane.
They defined the noitions of category, functors, and natural transformations
to get clear, abstract and pervasive notions to describe a variety
of settings and workings in mathematics. Category theory really
started to gain a life of its own from the late 1950's,
especially with notion of adjoint
functors (see the appendix) introduced by Daniel Kan in 1958.  
The language of category theory has now advanced into most domains
of pure mathematics. Since the 1990's it has also been 
comprehensively in use in theoretical computer
science. In the last 10-15 years Applied Category Theory has surfaced
\cite{ACT}. For instance, categorical notions are realized to be natural
in database and network modeling.

To put it in popular terms, category theory concerns the ``social life''
of mathematical objects. It is not a priori concerned with inner structure, but
rather how the objects relate to each other. But (just as with people),
how they relate reflects much on their inner structure. So, 
it also becomes a tool (a universal such) to understand inner structure.

\begin{example}
  Consider the mathematical notion of {\it group}.
  Groups form a collection of mathematical objects.
  Between any two groups $G_1$ and $G_2$ there are structure preserving maps,
  the {\it group homomorphisms}   $f : G_1 \pil G_2$.
  \begin{itemize}
  \item Morphisms may be composed. Given homomorphisms
    \[ f : G_1 \pil G_2, \quad g : G_2 \pil G_3 \]
    we get the composition $g \circ f : G_1 \pil G_3$.
  \item For each group $G$ there is an identity homomorphism
    $1_G : G \pil G$.
  \end{itemize}

  The homomorphisms fulfill two properties:
  \begin{itemize}
  \item If $h : G_3 \pil G_4$ is a third homomorphism, we have equality of
    the two compositions
    \[ h \circ (g \circ f), \quad (h \circ g) \circ f. \]
  \item The compositions $1_{G_2} \circ f = f \circ 1_{G_1} = f$.
  \end{itemize}
\end{example}  

\subsection{Categories}
Analogous situations, with objects and structure preserving maps that show how
the objects relate, occur all over. So, we take a large step up in
generality and abstraction to encompass this.

\begin{itemize}
\item  A category consists of a collection $\cC$ of {\it objects}.
(What they are we do not specify in the definition.
But a good list of examples gives us the concreteness needed to think
about them.)
\item Between any two objects $C_1$ and $C_2$ in $\cC $ there is a {\it set}
$\cC(C_1,C_2)$, the {\it morphisms} between $C_1 $ and $C_2$. If
$f \in \cC(C_1,C_2)$ we write $f : C_1 \pil C_2$.
\item For each three objects $C_1, C_2, C_3$ there is a map of sets
  \[\circ :  \cC(C_1,C_2) \times \cC(C_2,C_3) \pil \cC(C_1,C_3),
    \quad (f,g) \mapsto g \circ f   \]
\item For each object $C$ there is an element $\ben_C \in \cC(C,C)$,
  the {\it identity map}. 
\end{itemize}

\medskip 
The morphisms fulfill two properties:
\begin{itemize}
\item Associativity: Given
  \[f : C_1 \pil C_2, \, g : C_2 \pil C_3, \, h : C_3 \pil C_4  \]
  the compositions
  \[ h \circ (g \circ f), \quad (h \circ g) \circ f \]
    are equal.
  \item Composition with the identity map:
    $\ben_{C_2} \circ f = f \circ \ben_{C_1} = f$.
  \end{itemize}

  \begin{example}
    The most used category may be the category of sets, denoted $\set$.
    The objects are sets and for sets $X$ and $Y$ a morphism
    is simply a function $f : X \pil Y$.
   \end{example}

   That two objects are isomorphic means that they are
   ``structurally the same''. It seems one must somehow ``look inside''
   the objects. But this property can be described entirely by how
   they relate to each other, it is a  {\it categorical} notion.

   \begin{definition} Let 
     Let $f : C_1 \pil C_2$ be a morphism. This is an {\it isomorphism} if
     there is a morphism $g : C_2 \pil C_1$ such that $g \circ f = \ben_{C_1}$
     and $f \circ g = \ben_{C_2}$.
   \end{definition}

     Category theory was mainly developed as an {\it external language} for
  describing the general workings of how mathematical objects relate.
  But its fundamental status and significance also shows itself in how
  it may give {\it internal descriptions} of mathematical structures.

   Let us look at the simplest possible categories:
  \begin{itemize}
  \item[1.] Categories with a single object,
  \item[2.] Categories where for each pair of objects $a,b$ the
    set of morphisms $\cC(a,b)$ either has a single element, or is empty.
  \end{itemize}

   \subsection{Groups as categories}  
  \begin{example}
    Let a category have a single object $c$. The interesting part is
    then the morphisms $M = \cC(c,c)$. We have an identity morphism
    $1_c$ and we may compose:
    \[ \circ \, : \,  M \times M \pil M, \quad (f,g) \mapsto g \circ f. \]
    This fulfills
    \[ \ben_c \circ  f = f \circ \ben_c, \quad f \in \cC(c,c) = M \]
    and the associativity property
    \[ h \circ (g \circ f)  = (h \circ g) \circ f, \quad f,g,h \in M. \]
    This means precisely that $M$ is a {\it monoid} with $\circ$ as the
    product map.  If all the morphisms are
    isomorphisms, we get precisely the notion of {\it group}.

    This categorical view of a group in fact displays the basic idea that
    a group is the symmetries of some object. Compared to the modern abstract
    definition of a group the categorical viewpoint adds the classical
    intuition.
  \end{example}
Our first example was the category of groups. And now in fact we
  see that each single group itself may be considered a category.
  The inner structure of the object (group) is described as a category.

  \subsection{Preorders as categories}

    Consider a category whose objects form a set, and
    where  between any pair of objects $a,b$ there is {\it at most one}
    morphism.
    Write $a \leq b$ if there is a morphism.
    The existence of an identity morphism says that $a \leq a$ for
    every object $a$. If $a \leq b $ and $b \leq c$, there must
    be a composition, so $a \leq c$.

    The requirements of associativity of composition and on the identity
    morphism are automatic when there is at most a single morphism between
    two objects.
    Thus, such a category corresponds precisely to a set with a {\it preorder}.

    \begin{definition}
      A relation $\leq$ on a set $V$ is {\it preorder} if it is:

      \begin{itemize}
      \item {\bf Reflexive:} $a \leq a$ for every $a \in V$,
        \item {\bf Transitive:} $a \leq b$ and $b \leq c$ implies $a \leq c$.
        \end{itemize}
        We may have extra conditions.

\medskip \noindent a. The preorder is a {\it partial order},
      and $(V, \leq)$ is a {\it poset} if it is:
        \begin{itemize}
        \item {\bf Anti-symmetric:} $a \leq b$ and $b \leq a$ implies $a = b$.
        \end{itemize}
        In categorical language this says that each
        isomorphism class consists of a single
        object (element).
        
\medskip \noindent b. The preorder is an {\it equivalence relation} if it is:
        \begin{itemize}
        \item {\bf Symmetric:} $a \leq b$ implies $b \leq a$.
        \end{itemize}
        In categorical language this is a {\it groupoid}: Every morphism
        is invertible, so is an isomorphism. 

\medskip \noindent c. The preorder is {\it total} if:
        \begin{itemize}
        \item {\bf Comparability:}
          For any $a, b \in V$ either $a \leq b$ or $b \leq a$ (or both).
          \end{itemize}

\medskip \noindent d. The preorder is {\it discrete}
          if $a \leq b$ implies $a = b$, so any two distinct elements
          are not comparable.

\medskip \noindent e. The preorder is {\it coarse} 
          if $a \leq b$ for any $a, b \in V$.           
        \end{definition}

A preorder $\leq$ induces an equivalence relation $\sim$
where $a \sim b$ if
$a \leq b$ and $b \leq a$. The equivalence classes of $\sim$ are
the {\it bubbles} of the preordered set $(V,\leq)$.
Letting $\overline{V}$ be these equivalence classes, the preorder
induces a poset $(\overline{V},\leq)$.

\begin{example} \label{ex:cat-tre}
The following displays the preorders with two elements. Greater elements are
drawn above smaller.

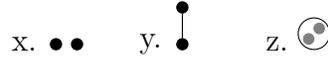
\begin{equation*} \topre
\end{equation*}
The first is the discrete poset on two elements (two not comparable elements),
the second has two elements
where one is bigger than the other, and the third is the coarse preorder with
two elements. 

The following are the nine preorders with three elements:
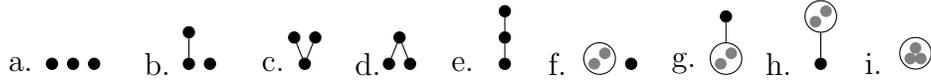
\begin{equation*} \trepre
\end{equation*}

\begin{itemize}
\item[$\circ$] Discrete preorder: a. Coarse preorder: i.
  \item[$ \circ $] Partial orders: a, b, c, d, e.
  \item[$ \circ $] Total preorders: e, g, h, i.
  \item[$ \circ $] Equivalence relations: a, f, i. Preorders a, b, c, d, e induce
    the equivalence relation a. Preorders f, g, h induce f.
  \item[$ \circ $] Associated partial order of f is x, of g, h is y.
  \end{itemize}
  
\end{example}

The set of all preorders on $V$ will be denoted by $\Pre(V)$.
This set has itself a partial order
${\bf \preceq}$ where $\leq_1 \, \bpreceq \, \leq_2$
if $a \leq_1 b$ implies $a \leq_2 b$ for every $a,b \in V$.

When $V$ has two elements, $\Pre(V)$ has $4$ elements. 
When $V$ has three elements, $\Pre(V)$ has $29$ elements. The group
of permutations of $V$ acts on $\Pre(V)$, and the orbits correspond
precisely to the $9$ isomorphism classes of preorders listed above.
    


   \medskip   
  From these two most basic examples of categories (both occur
in the founding article \cite{EiML} from 1945), 
  we see the fundamental status of {\it symmetry} (groups) and of {\it order}
  (preorders) in mathematics. Symmetry and groups are perhaps the most
  studied and utilized themes in mathematics. This review concerns
  {\it order}.
  We give examples where order structures are not apparent on the face of it,
  but where
  a deeper investigation reveals them to be governed by underlying order
  structures. 

\section{Directed graphs, categories, and preorders}
\label{sec:digraph}

An even simpler type of order structure is provided by the notion of
{\it directed graph}. This induces both a preorder
and an enrichment of this preorder, a category, as we now shall see.

Let $V$ be a set and consider
  an ordered pair of maps $s,t : E \pil V$. This is equivalent to give
  a map $E \overset{s \times t}{\lpil} V \times V$. This is simply a directed
  graph $Q$ (with multiple edges allowed): The edges are $E$, the vertices
  $V$ and for an edge $e \in E$, $s(e)$ is the start vertex and $t(e)$ its
  end vertex. 

  It is tempting to think this could be a category with objects $V$ and
  morphisms from $v$ to $w$ to be edges from $v$ to $w$.
  At least an edge $e$ gives an {\it ordered} pair $(s(e), t(e))$. 
  This is not a category as it does not contain a prescription
  to compose two edges $e,f$ and get a new edge $g$. In particular
  it is also not a preorder.
  But it seems it should be the germ of a category and of a preorder.

  Consider paths of edges in this directed graph $Q$:
  \[ e_1, e_2, \ldots, e_n \]
  where $t(e_{i-1}) = s(e_i)$.
Letting $v = s(e_1)$ and $w = t(e_n)$ this is a {\it path} from $v$ to $w$.
At each vertex $v$ we consider there to
  be an empty path $1_v$.
  Paths $a$ and $b$ may be composed by concatenation, giving a product
  $ab$ {\it if} the
  end vertex of $a$ is the same as the starting vertex of $b$.
  This gives a partial associative product on the set of paths $Q^*$
  Thinking of a path from $v$ to $w$ as a morphism from $v$ to $w$, this becomes
  a category, the {\it free category} $\path(Q)$ generated by the
  directed graph $Q$. This gives a functor (see the Appendix for this notion),
  from directed graphs to small categories (small category: a category
  whose objects form a set).

  \begin{example} \label{ex:dig-dig}
  Consider the digraph $Q$:
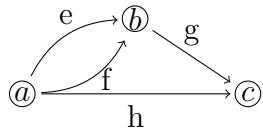
\[    \digraph \]
It has the following paths:
\[ 1_a, \, 1_b, \, 1_c, \, e, \, f, \, g, \, h, \, eg, \, fg. \]
In the category $\path(Q)$ there are three morphisms from $a$ to $c$. These
are $h, eg$ and $fg$.
    \end{example}

  Conversely if $\cC$ is a small category, we get a directed graph by
  letting $E$ be the disjoint union $\sqcup_{a,b \in \cC} \Hom(a,b)$.
Denoting this forgetful functor by $U$, we
 have the two functors on the left below:
  \begin{equation} \label{eq:dig-adj}
    \Digraph \bihom{\path}{U} \smCat  \bihom{\preo}{\inc} \Preo.
    \end{equation} 
 These form an adjunction (see Appendix), with $\path $ left adjoint to $U $: 
  \[ \Hom_{\smCat}(\path(Q), \cC) = \Hom_{\Digraph}(Q, U \cC). \]
  On the right side of \eqref{eq:dig-adj}
  we have a functor $\preo$ sending a small category $\cC$ to
  a preorder $\pre(\cC)$ with the same objects and where $a \leq b$ iff there is
  some morphism $a \pil b$ in $\cC$. The functor $\inc$ is simply the inclusion
  functor. These form an adjunction with $\preo$
  left adjoint to $\inc$:
  \[ \Hom_{\Preo}(\preo(\cC), P) = \Hom_{\smCat}(\cC, \inc(P)). \]
  Putting these two adjunctions together we get the adjunction:
  \[ \Digraph \bihom{\preo \, \circ \, \path }{U \, \circ \, \inc}
    \Preo.\]  
 For a low-treshold introduction to categorical notions
  starting from preorders, see the first sections of \cite{ACT}.

\section{Topologies}
\label{sec:top}

Most of us have likely seen the basic definition of a topology on a
set $X$ in
terms of open sets. Usually in topology, the set $X$ 
is infinite, and one puts extra conditions on
the topology to develop richer theory. For instance, one may require each
point of $X$ to have an open set neighbourhood which looks like an
open set in some $\RR^n$. This gives the notion of topological { manifold}.
The modern notion of a topology was in an equivalent form introduced by
Felix Hausdorff in his influential book {\it Grundzüge der Mengenlehre}
in 1914:

\begin{definition} \label{defi:top-top}
Let $\cT$ be a family of subsets of a set $X$. This is a {\it topology}
on $X$ if:
\begin{itemize}
\item $\emptyset \in \cT$ and $X \in \cT$,
\item If $U, V \in \cT$, then $U \cap V \in \cT$,
\item If $\{U_i\}$ is any family of sets with $U_i \in \cT$, then
their union $\cup_i U_i \in \cT$.
\end{itemize}
The elements of $\cT$ are the {\it open sets} of the topology.
A subset of $X$ is {\it closed} if its complement set is open.
\end{definition}

No order structure is apparent on $X$ in this definition. Indeed,
any sphere with the standard notion of open subsets, forms
a topology, and a sphere has a priory no natural order structure. 
Now assume $X$ is finite. Digging into
what a topology on $X$ is, gives a surprise. First an example
with finite $X$:

\begin{example} \label{eks:top-pre}
  Let $(X,\leq)$ be a set with a preorder. Am {\it up-set} $U$ of this
  preorder, is a subset
  of $X$ such that $x \in U$ and $x \leq y$ in $X$ implies $y \in U$.
  Both $X$ and the empty set are up-sets of $X$. We also see that taking
  an arbitrary union or intersection of up-sets of $X$ we still get a
  up-set.
  Hence up-sets fulfill the requirement of being
  a topology.
\end{example}

Now comes an amazing correspondence and result. 
Given a topology
  $\cT$ on a set $X$.  For $x,y \in X$, define $x \leq y$ if
  any open subset $U$ containing $x$ also contains $y$.
  Clearly i) $x \leq x$ and ii) $x \leq y$ and $y \leq z$ implies $x \leq z$. 
So this gives a preorder. With Example \ref{eks:top-pre} we thus have correspondences:

\[ \Pre(X) \bihom{\topt}{\preo} \Top(X). \]

\begin{theorem} \label{top:thm-ft}
  For a finite set $X$, these correspondences are inverse bijections.
  Thus, to give a topology $\cT$ on $X$ is {\bf equivalent} to give
  a {\bf preorder} $\leq$ on $X$. The topology is then precisely the up-sets
  of the preorder.
\end{theorem}

\begin{proof} 
If a preorder $\leq $ induces a topology $\cT$, it is immediate that
this topology again induces the preorder $\leq$, so $\preo \circ \topt = \id$.

Let a topology $\cT$ induce a preorder $\leq$. This preorder induces
a topology $\cT^\prime$.
We show $\cT^\prime = \cT$. 
Let $U$ be an open subset of $X$.
If $x \in U$ and $x \leq y$, then $y \in U$ by definition of the preorder.
Hence $U$ is an up-set for $\leq$ and $U \in \cT^\prime$.
Thus $\cT \sus \cT^\prime$. 

Conversely let $U^\prime \in \cT^\prime$. We show $U^\prime \in \cT$.
For each $x \in U^\prime$, let
$\ux$ be the intersection of all $U$ in $\cT$ containing $x$. Since
$X$ is finite, this is a finite intersection and so $\ux$ is in $\cT$.
As $\cT \sus \cT^\prime$, $\ux$ is also in $\cT^\prime$.

Let $y \in U_x$. By definition of the preorder $\leq$, we have $x \leq y$.
But then $y \in U^\prime$ as $x \in U^\prime$ and $U^\prime$ is an up-set.
Hence $\ux \sus U^\prime$. So
$U^\prime = \cup_{x \in U^\prime} \ux$ and so $U^\prime \in \cT$. 
\end{proof}

It is remarkable that the notion of partially ordered set (poset) was
also introduced by Hausdorff in the first decade of the 1900's. 

  In Definition \ref{defi:top-top} we may require that for {\it any} family
  $\{U_i\}$, finite or infinite, of elements of $\cT $,
  the intersection $\cap U_i$ is
  in $\cT $. Such spaces were introduced by Pavel Alexandrov in 1937
  \cite{ASpace},
  and he called them ``discrete spaces''
  (a term nowadays used for the topology
  consisting of all subsets). In his 1986 book \cite{StSp} Peter Johnstone
  named them {\it Alexandrov spaces} or {\it Alexandrov topologies}. 
  The above proof shows that for any set $X$, Alexandrov topologies
  on $X$ correspond bijectively to preorders on $X$.
  This simple fact seems
  first explicitly stated in the influential 1966 papers
  \cite[Section 4]{McCo} and \cite[Thm.2.6]{Steiner}.
  But Alexandrov already in
  \cite{ASpace} observed that $T_0$-Alexandrov spaces
  (a space is $T_0$ if for any two points there is an open subset containing
  one of them and not the other) correspond bijectively to partial orders.
  The recent book \cite{May-FC} shows finite topologies as a rich source
  even from the viewpoint of topologists. \cite{HHfinite} illustrates
  how they are used as a good tool in teaching basic abstract concepts
  to undergraduates. 

  \medskip
{\it Birkhoff's theorem} \cite[Section 8]{DaPr} dates from 1937
and is very closely related to Theorem \ref{top:thm-ft}. 
  They are essentially the same. Birkhoff's theorem says that a 
  finite (abstract) distributive lattice is isomorphic to the
  lattice of up-sets in a partially ordered set $P$. This gives a
  bijection between
  finite distributive lattices and finite posets.

  A lattice of subsets of a finite set $X$, with union as join, and meet as
  intersection is the same as a topology on $X$.
  Such a lattice is also
  automatically distributive.
  So, Birkhoff's theorem produces a poset. It is in fact the
  poset $(\oX,\leq)$ induced on the set $\oX$ of bubbles of
  the preorder $X$. Birkhoff worked from the perspective of lattice
  theory, culminating in his 1940 monograph {\it Lattice theory} \cite{BiLa}.
  
  Although complete lattices of sets, i.e. lattices of subsets with
  {\it arbitrary unions and intersections} (which is the same as an
  Alexandrov topology) are discussed several places in the standard
  book \cite{DaPr}, it is surprising that it is not mentioned there 
  that this corresponds precisely to a preorder on the ambient set.

\section{Associative algebras}
\label{sec:ass}

An algebra over a field $k$ is a vector space $A$ over $k$
with a multiplication operation, which is simply a
map $\mu : A \te_k A \pil A$, and
we typically write the image $\mu(a,b)$ as $a \cdot b$ or simply $ab$.  

Various requirements may be put on $\mu$. It could be symmetric,
skew-symmetric, associative, fulfill a Jacobi-like identity
and so on. See \cite{LV} for a modern understanding with a plethora of
natural multiplication structures which can be put on a vector space.

A most basic one is being associative:

\[ a(bc) = (ab)c \] or, in terms of the map $\mu$, the following
diagram commutes:

\[ \xymatrix{ A \te A \te A \ar[d]_{\mu \te \ben}
\ar[r]^{\ben \te \mu} & A \te A \ar[d]^{\mu} \\
A \te A \ar[r]^{\mu} & A.}
\]

\begin{definition}
  Let $A$ be a $k$-vector space with a multiplication $\mu$. This is
  an {\it associative algebra} if:
  \begin{itemize}
  \item The multiplication is associative,
  \item There is an identity element $1$ such that $1\cdot a = a \cdot 1 = a$
    for $a \in A$.
  \end{itemize}

\end{definition}

\begin{example}
  Let $M_n$ be the $n \times n$-matrices over a field $k$. This is
  an associative algebra under matrix multiplication.

\end{example}

\begin{example} 
Let  $X$ be a set, an alphabet, and $X^*$ all words on
$X$, including the empty word. There is a product which is concatenation
of words:
\[ xyyz \cdot zyxz = xyyzzyxz. \]
This makes $X^*$ into a monoid, the free monoid on the set $X$.
\end{example}

  To a monoid $M$ we may form the associated monoid algebra $kM$,
  a vector space with basis $M$. 
The associated monoid algebra $kX^*$, usually written
$k\langle X \rangle$, is an associative algebra with basis the words $X^*$.
It is the {\it free associative algebra} on the vector space $kX$.








\subsection{Path algebras}
Let $Q$ be a directed graph $E \mto{(s,t)} V$ with $V$ finite (also
called a quiver).
Recall from Section \ref{sec:digraph}
the free category $\path(Q)$ generated by $Q$. 
Let $\kQ$ be vector space whose basis is the
set of paths $Q^*$ in $Q$. Recall that $Q^*$ are the morphisms in
$\path(Q)$.
We extend the partial product on $Q^*$ to
  a full product on $\kQ$ by $ab = 0$ if $a$ and $b$
  cannot be concatenated. 
  Now there {\it is} an identity element, simply the sum $1 = \sum_{v \in V} 1_v$.
  We then get the {\it path algebra} $\kQ$, an
  associative algebra. If $Q$ does not have directed cycles, this is a
  finite-dimensional algebra. If $Q$ has directed cycles, $\kQ$
  is infinite-dimensional as a vector space.
  
 Given a set $R$ of pairs of paths
  $(p^1_i, p^2_i)$ 
  where the paths in each pair have the same start vertex and the same
  end vertex, we may form a quotient category $Q^*/R$ where the
  paths $p^1_i$ and $p^2_i$ are declared to be equal. The objects
  are still $V$ and the morphisms are the equivalence classes of paths
  generated by the equivalence relation $t p^1_i s \equiv t p^2_i s$,
  for $t, s$ are paths in $Q$. 
  In particular, if we for each pair of vertices $v,w$ declare
  all paths between them to be equal, we get the preorder generated by
  the directed graph.

  \medskip
  More generally we can consider the path algebra $\kQ$ and
  a subset $R \sus \kQ$ (relations in $\kQ$). Let $(R)$
  be the two-sided ideal in $\kQ$ generated by $R$. This
  gives a quotient algebra $\kQ/(R)$.

  A relation $r \in R$ is a linear combination of paths $r = \sum_i \alpha_i
  P_i$. By multiplying with an idempotent $1_v$ on the left and an
  idempotent $1_w$ on the right, we get a new relation still in $(R)$.
  If this is non-zero
  every $P_i$ in this new sum have the same start vertex $v$ and the same end
  vertex $w$. Thus, we may assume $(R)$ is generated by linear combinations
  of paths with the same start vertex and the same end vertex. 

  \begin{example} Consider the directed graph of Example \ref{ex:dig-dig}
    and in its path
    algebra the relation:
    \begin{equation} \label{eq:ass-rel} h  + 3f - eg.
      \end{equation} 
    Multiplying with $1_c$ on the right we get the relation $h - eg$, and
    multplying on the right with $1_b$, we get the relation $3f$. Hence the
    ideal generated by \eqref{eq:ass-rel} is the same as the ideal
    generated by $3f$ and $h-eg$.
  \end{example}

  Let $I_Q$ be the ideal generated by the paths of length $\geq 1$.
  We may assume $(R) \sus I_Q^2$, the paths of length $\geq 2$. Otherwise
  the quotient algebra $k\langle Q \rangle/(R)$ could be written
  as $k\langle Q^\prime \rangle/(R^\prime)$ for a subquiver $Q^\prime$ of $Q$.
  When $(R) \sus I_Q^2$ we can recover $Q$ from the algebra
  $k\langle Q \rangle/(R)$. That $k\langle Q \rangle/(R)$ is
  finite-dimensional is equivalent to $I_Q^m \sus (R)$ for some $m$. When
  \[ I_Q^m \sus (R) \sus I_Q^2 \] for some $m$, $(R)$ is {\it admissible}. 

  \medskip
  The take-away is that directed graphs, together with relations between
  paths give a large reservoir of associative algebras. Surprisingly for
  finite-dimensional algebras, they are essentially all.

  \subsection{Finite dimensional algebras and module categories}

 \begin{definition}
  A $k$-vector space $M$ is a (left) module over $A$ if there is multiplication
  map $\nu : A \te M \pil M$ such that (we write $am := \nu(a,m)$):
  \begin{itemize}
  \item $a(bm) = (ab)m$
  \item $1m = m$
  \end{itemize}
\end{definition}

If $A$ is an associative $k$-algebra, we denote by $\mod_A$ the category of
finite dimensional modules. The simplest such module
categories arise when $A = k$. The modules are then simply the finite
dimensional
$k$-vector spaces. It has only one indecomposable module, the one-dimensional
vector space $k$. Every finite dimensional vector space is isomorphic
to $k^n$ for some $n$.

Considering the matrix algebra $M_n$ the situation is similar. We consider
$k^n$ as a column vector and it is a module over $M_n$ by matrix multiplication
on the left. Then $k^n$ is a simple $M_n$-module, and it is actually the
unique simple $M_n$-module. Every finitely
generated module over $M_n$ is a direct sum of this simple module.
(Contrast this
to representations of the {\it Lie group} $GL(n)$.)
The module categories $\mod_k$ and $\mod_{M_n}$ are in fact equivalent
categories. Two associative algebras with equivalent module categories
are said to be {\it Morita equivalent}. Thus the $M_n$ are Morita equivalent
algebras for $n \geq 1$.  \cite[Section 18]{Lam} gives an excellent
presentation of Morita equivalence.

From the definition of associative algebras there is no apparent order
on such algebras. (Nevertheless, a hint may be non-commutativity, that
$ab$ in general does not equal $ba$.) 
The following shows however that basic order relations underlie
the structure of
finite dimensional associative algebras, \cite[Theorem II.6.16]{Sko} or
\cite[Theorem 3.7]{ASS},
or more implicit \cite[Cor.II.2.6 and III.Cor.1.10]{ARS}.

\begin{theorem} Let $k$ be algebraically closed.
  Any finite dimensional associative $k$-algebra where $k$ is a field,
  is Morita equivalent to a {\bf path algebra} with relations $\kQ/(R)$ where
  $(R)$ is an admissible ideal. 
\end{theorem}

Morita equivalence dates back to \cite{Mo58}. The ideas from that article
were then widely circulated in lectures by Hyman Bass
in the 1960's. 
Due to above theorem, in the representation theory of finite-dimensional
algebas one needs only work with representations of path algebras
with admissible relations.
The representation theory of finite-dimensional associative algebras
took off especially with the work of Maurice Auslander and Idun Reiten
in the 1970's with what is now called Auslander-Reiten theory
\cite{ARS}, and has since been
a very active area. A finite-dimensional
associative algebra may be considered
a category with one object and with morphisms enriched in finite-dimensional
vector spaces.
The setting for the above theory is now the greater generality of 
(finite) categories
enriched in finite-dimensional vector spaces.

\section{Subgroups of $\glv$}
\label{sec:glv}

Here we consider the group $\gln$ of invertible
$n \times n$-matrices and its subgroups which contain the diagonal
matrices. We will see a striking correspondence between such subgroups
and preorders on a set of cardinality $n$. 

We assume in this section that $k$ is an algebraically closed field
of characteristic zero. Let $k^*$ be the non-zero elements of $k$. 
Let $V$ be a vector space of dimension $n$ and $\glv$ the group if
invertible linear maps $g : V \pil V$.
If $H$ is a subgroup (not necessarily normal) of a group $G$ we use the
notation $H \triangleleft G$. 

\subsection{Subgroups containing a maximal torus}
If we fix an ordered basis for $V$:  $v_1, v_2, \ldots, v_n$, we may
represent $\glv$ as invertible matrices $\gln$. The
diagonal matrices $D$ is isomorphic to $(k^*)^n$, and an element
$\bfu = (u_1, \ldots, u_n) \in D$ acts on $V$ by:
\[ \sum_1^n k_i v_i \overset{u}{\mapsto} \sum_1^n k_i u_i v_i. \]
A subgroup of $GL(V)$ isomorphic to $(k^*)^n$ is called a
{\it maximal torus}. There are many such. For each $g \in \glv$, the
conjugate $T = g^{-1} D g$ is a maximal torus, and every maximal torus
is such a conjugate, \cite[Cor. 21.3A]{Hum}.

Now given a torus $T \iso (k^*)^n$ acting on a finite dimensional
vector space $W$. The irreducible representations of $T$ are one-dimensional.
So there is a set $X$ of vectors in $W$ such that $W = \oplus_{x \in X} V_x$
where $V_x = \langle x \rangle$ is the one-dimensional
vector space generated by $x$. When $T$ is a maximal torus in $\glv$
acting on $V$, we have $V = \oplus_{x \in X} V_x$ where the choice of each
$x$ is determined up to a scalar multiple. The group $\glv$ may be
identified with (unordered) matrices $(a_{xy})_{x,y \in X}$, and
$T$ may then be identified with the diagonal matrices:
$a_{xy} = 0$ for $x \neq y$ and $a_{xx} \neq 0$. Write then $\glx$
for $\glv$. 
If we put a total order on $X$ these become the (ordered) matrices $\gln$ and
the (ordered) diagonal matrices.

Now let $X$ have a preorder $\geq$. Each $x \in X$ determines an
up-set $\ux = \{ y \in x \, | \, y \geq x \}$. It generates a subspace
$\langle \ux \rangle$ of $V$.

\begin{definition} \label{defi:glv-glxo} Given a set $X$ with a preorder
  $\geq$. Denote by 
  $\glxo{}$ the subgroup of $\glx$ consisting of all $g \in \glx$
  such that $g. \langle \ux \rangle \sus \langle \ux \rangle$
  for each $x \in X$. 
\end{definition}

\begin{example}
  On the set $X = \{x,y,z\}$ there are $29$ preorders. To display
  $GL_\geq(X)$ as a matrix we put $x,y,z$ in sequence. (This may be
  seen as a total order, the display order, but is unrelated to the
  order $\geq$.) The six possible total orders on $X$, giving diagram
  e. in Example \ref{ex:cat-tre}, give the six matrix groups below
  (with * if an entry may be nonzero) 
  :
  \begin{equation*}
\overset{x > y > z}{ \left [\begin{matrix} * & * & * \\ 0 & * & * \\ 0 & 0 & * \end{matrix}\right ]},
    \overset{x > z > y}{ \left [\begin{matrix} * & * & * \\ 0 & * & 0 \\ 0 & * & * \end{matrix}\right ]},
\overset{y > z > x}{ \left [\begin{matrix} * & 0 & 0 \\ * & * & * \\ * & 0 & * \end{matrix}\right ]},
    \overset{y > x > z}{ \left [\begin{matrix} * & 0 & * \\ * & * & * \\ 0 & 0 & * \end{matrix}\right ]},
\overset{z > x > y}{ \left [\begin{matrix} * & * & 0 \\ 0 & * & 0 \\ * & * & * \end{matrix}\right ]},
    \overset{z > y > x}{ \left [\begin{matrix} * & 0 & 0 \\ * & * & 0 \\ * & * & * \end{matrix}\right ]}.
        \end{equation*}

If we consider the nine isomorphism classes of  preorders a. $\ldots$ i. of
Example \ref{ex:cat-tre}, with a choice of labelling of $x,y,z$ we have the
following subgroups representing them:
 \begin{equation*}
\overset{a.}{ \left [\begin{matrix} * & 0 & 0 \\ 0 & * & 0 \\ 0 & 0 & * \end{matrix}\right ]},
    \overset{b.}{ \left [\begin{matrix} * & * & 0 \\ 0 & * & 0 \\ 0 & 0 & * \end{matrix}\right ]},
\overset{c.}{ \left [\begin{matrix} * & 0 & * \\ 0 & * & * \\ 0 & 0 & * \end{matrix}\right ]},
    \overset{d.}{ \left [\begin{matrix} * & * & * \\ 0 & * & 0 \\ 0 & 0 & * \end{matrix}\right ]},
\overset{e.}{ \left [\begin{matrix} * & * & * \\ 0 & * & * \\ 0 & 0 & * \end{matrix}\right ]},
\end{equation*}

 \begin{equation*}
\overset{f.}{ \left [\begin{matrix} * & 0 & 0 \\ 0 & * & * \\ 0 & * & * \end{matrix}\right ]},
    \overset{g.}{ \left [\begin{matrix} * & * & * \\ 0 & * & * \\ 0 & * & * \end{matrix}\right ]},
\overset{h.}{ \left [\begin{matrix} * & * & * \\ * & * & * \\ 0 & 0 & * \end{matrix}\right ]},
    \overset{i.}{ \left [\begin{matrix} * & * & * \\ * & * & * \\ * & * & * \end{matrix}\right ]},
        \end{equation*}
        The reader may convince herself that given a type of matrices above,
        multiplying
        two such matrices gives the same type.
      \end{example}

Recalling the partial order $\bsucceq$ on $\Pre(X)$, Section \ref{sec:cat}, 
some observations are:

\begin{itemize}
\item $\glxo{}$ is a closed subgroup of $\glx$. It identifies as all
  invertible matrices $(a_{xy})$ such that $a_{xy} = 0$ if not
  $x \leq y$.
\item If $\geq_D$ is the discrete preorder, then $\glxo{D}$ is the
  maximal torus $T$.
\item If $\geq_C$ is the coarse preorder, then $\glxo{C}$ is the
  whole general linear group $\glx$.
\item If $\geq_L$ is a total order $x_1 >_L x_2 >_L \cdots >_L x_n$, then
  using this also as the display order, 
  $\glxo{L}$ identifies
  as the upper triangular matrices $B$ of $\gln$.
  \item If $\geq_1 \, \bsucceq \, \geq_2$ are two preorders on $X$, then
  $\glxo{1} \triangleright \glxo{2}$.
\end{itemize}

\subsection{Correspondence between subgroups and preorders}
Fix a maximal torus $T \sus \glv$ with corresponding basis $X$ of $V$.
Define a $T$-subgroup of $GL(V)$ to be a subgroup of $GL(V)$ containing
$T$. The correspondence given in Definition \ref{defi:glv-glxo}
gives a functor:
\[ \Pre(X) \overset{GL_{( )}(X)}{\lpil} T\!-\!\text{subgroups of } \glx. \]
Conversely given a subgroup $H$ with $T \triangleleft H \triangleleft \glv$, let $\cT \sus P(X)$
(a subset of the power set on $X$)
consist of all subsets $Y \sus X$ such that $g.\langle Y \rangle  \sus
\langle Y \rangle$ for $g \in H$.

\begin{claim} $\cT$ is a topology on $X$. \end{claim}

\begin{proof} 
  Let $Y_1, Y_2 $ be in $\cT$. Then
  \[\langle Y_1 \cap Y_2 \rangle = \langle Y_1 \rangle \cap \langle Y_2 \rangle,
    \quad
    \langle Y_1 \cup Y_2 \rangle = \langle Y_1 \rangle +  \langle Y_2 \rangle.\]
  The right sides of these equalities are invariant for $g \in H$, and
  so also the left sides, so both $Y_1 \cap Y_2$ and
  $Y_1 \cup Y_2$ are in $\cT $.
\end{proof}

By Section \ref{sec:top}, $\cT $ corresponds to a preorder $\geq $ in $X$,
and $\cT $ is the set of all up-sets of $\geq $.
Thus, we get functors:
\begin{equation} \label{eq:glv-ad}
  \Pre(X) \bihom{GL_{( )}(X)}{\pre}  T\!-\!\text{subgroups of } \glx,
\end{equation}
which are really just order-preserving maps between posets.
As
\[ \pre(H) \, \bpreceq \,\,   \leq  \quad \text{iff} \quad H \triangleleft
\glxo{}, \]
this is an adjunction, or more specifically a Galois correspondence, and
$\pre $ is a left adjoint to $GL_{( )}(X)$. The composite $\pre \circ GL_{( )}(X)$ is the identity.

\begin{example}
  For each bijection $\pi : X \pil X$, consider invertible matrices $(a_{x,y})$
  such that $a_{x,y} = 0$ if $y \neq \pi(x)$. These must have each
  $a_{x,\pi(x)} \neq 0$. Let $A$ be such a  matrix, a $\pi$-{\it matrix}.
  Let $\sigma$ be another permutation of $X$, and $B$ a $\sigma$-matrix.
  Then the product $AB$ is a $\sigma \circ \pi$-matrix. Hence these matrices
  form a subgroup $N$ of $GL(X)$ containing the diagonal matrices $D$.
  In fact, $N$ is the normalizer of $D$ in $GL(X)$, and $N/D$ identifies
  as the finite group $S(X)$ of permutations of $X$.
  The connected components of the algebraic group $N$ are in bijection
  with $S(X)$, one component for each permutation.

  The preorder associated to $N$ is the coarse preorder.
  So, the correspondences \eqref{eq:glv-ad} are far from being bijections.
\end{example}

Note that the image of $GL_{( )}(X)$ in \eqref{eq:glv-ad} is always
a connected algebraic subgroup. In fact each $\glxo{}$ is
an open subset of an affine space. This suggests restricting
to connected $T$-subgroups.
\begin{equation*} 
  \Pre(X) \bihom{GL_{( )}(X)}{\pre}  \text{Connected }
  T\!-\!\text{subgroups of } \glx,
\end{equation*}

Z.I. Borevich 
\cite{Bor} describes all $T$-subgroups of $GL(X)$. 
An immediate consequence is the following.

\begin{theorem} Let $k$ be an algebraically closed field.
  The correspondences \eqref{eq:glv-ad} above are inverse bijections between
  connected algebraic
  $T$-subgroups of $GL(X)$ and {\bf preorders} on the set $X$.
\end{theorem}

For the description of {\it all} $T$-subgroups of $GL(X)$,
to each preorder $\geq$ let $N_\geq$ be the normalizer of $\glxo{}$.
Then $N_\geq/\glxo{}$ is a finite group. For a $T$-subgroup $H$ of $GL(X)$
its identity component contains $T$ and will thus be some group $\glxo{}$. 
Borevich shows that such subgroups $H$ are then in bijection with subgroups of
$N_\geq/\glxo{}$.
Total preorders $\geq$ correspond to parabolic subgroups $\glxo{}$
and in this case the normalizer $N_\geq$ actually equals $\glxo{}$.

\begin{remark} \cite{Bor} classifies all $T$-subgroups of $GL(V)$
  under the assumption that the field has $\geq 7$ elements.
  We consider here algebraic subgroups. It is only when the field is
  algebraically closed that there is a bijection between (irreducible) varieties
  and prime ideals. So, we state the above with this assumption.
  \end{remark}


\section{Ideals in polynomial rings}
\label{sec:pol}

Polynomial rings $k[x_1, \ldots, x_n]$ and ideals $I$ in polynomial
rings are the classical and fundamental algebraic objects in
algebraic geometry and commutative algebra.
Ideals may be degenerated to simpler ideals. The question we address here
is: What are the most degenerate ideals in a polynomial ring?

Let $\AA^n_k$ be the affine $n$-space, consisting of all $n$-tuples
$(a_1, \ldots, a_n)$ where the $a_i \in k$. 
To the ideal $I$ we associated its geometric object, the algebraic set
\[ V(I) = \{ \bfa \in \AA^n_k \, | \, f(\bfa) = 0 \text{ for every } f \in I\}.
\]
When $k$ is algebraically closed, Hilbert's Nullstellensatz says that this
is a bijection between algebraic sets in $\AA^n_k$ and {\it radical} ideals in
$k[x_1, \ldots, x_n]$. The ideal $I$ is radical if $f^n \in I$ implies
$f \in I$. An algebraic set which is {\it irreducible}, i.e.
not the union of two
proper algebraic subsets is a {\it variety}. Varieties are in bijection with
prime ideals in $k[x_1, \ldots, x_n]$.

\medskip
\subsection{Families of ideals and algebraic sets}
Varieties and ideals come in families. Consider the ideal in
$k[x,y]$ generated by the polynomials
\[ p = x^2 + 2xy+y, \quad q = y^2 - x - y . \]
Each of them defines a conic in the affine plane $\AA^2_k$ and
their intersection is  a set of four points.

We introduce a new parameter $t$, and perform the coordinate change
\[ x \mapsto \frac{1}{t^3}x, \quad y \mapsto \frac{1}{t^2} y. \]
The polynomials then become:
\[p_t = \frac{1}{t^6}x^2 + 2\frac{1}{t^5} xy + \frac{1}{t^2}y, \quad
  q_t = \frac{y^2}{t^4} - \frac{1}{t^3}x - \frac{1}{t^2}y. \]
They generate a family of ideals as $t$ varies. As long as $t  \neq 0$,
we may multiply with $t^6$ and $t^4$ and they generate the same ideals as:
\[ \overline{p_t} = x^2 + 2txy + t^4y,
  \quad \overline{q_t} = y^2 -tx-t^2y. \]
Now let $t \pil 0$. We then get the degenerate ideal generated by:
\begin{equation} \label{eq:pol-xy} x^2, y^2. \end{equation}
This is not a radical ideal anymore. Its zero set is now only the origin.
Geometrically it is nevertheless fruitful to
think of this as the origin ``fattened'' up
to a fat point of size $4$.
The ideal generated by $p$ and $q$ has now been {\it degenerated} to
the ideal $I$ generated by $x^2, y^2$.

\medskip
\noindent {\it Question.} Is the ideal $I$ as degenerate as
it gets, or is it possible to degenerate it still further?

\medskip
It turns out one can degenerate even more.

\medskip
\subsection{Monomial ideals and preorders}
Before proceeding with the above example, let us pause and discuss
monomial ideals. For a set $X$ let $k[X]$ be the polynomial ring where
the variables are the elements of $X$.
A monomial ideal in $k[X]$ is an ideal generated by monomials.
To each monomial ideal in $k[X]$ there is associated
a distinguished preorder on $X$, as we now explain.

Define a relation $R \sus X \times X$ by $(x,y) \in R$
if the following holds: 
Whenever $xm$ is a monomial in $I$ (a monomial with $x$ as a factor),
then $ym$ is also a monomial in $I$.
Clearly this relation is reflexive: All $(x,x) \in R$. It is also
transitive: If $(x,y)$ and $(y,z)$ are in $R$, then $(x,z) \in R$.
Som this is a preorder, and we get a map:

\begin{equation} \label{eq:pol-montopre}
  \text{Monomial ideals in } k[X] \lpil \Pre(X).
\end{equation}

The explicit formulation of this map is likely new.
But closely related is \cite{FMS},
where the point of view is
to start with a partial order $P$, and investigate monomial ideals
whose associated preorders are $\succeq P$.

For the examples below note that the symmetric group $S(X)$ of $X$,
consisting of the permutations of $X$, acts on the set of monomials in
$k[X]$ and so also on its set of monomial ideals. The symmetry group
of a monomial ideal $I$ is its stabilizer, i.e. the permutations $g \in S(X)$
such that $g.I = I$. 

\begin{example}
  The preorder associated to $I = (x^2, y^2)$
  is the discrete preorder. More generally the preorder associated to
  $I = (x^2 \, | \, x \in X)$ is the discrete preorder. However,
  the symmetry group of this ideal is the full group $S(X)$ of
  permutations of $X$. 
\end{example}

\begin{example}
  The preorder associated to the power ideal $I = (x \, | \, x \in X)^r$
  is the coarse preorder where $x \leq y$ for all $x,y \in X$. 
  The symmetry group of this ideal is the full group $S(X)$ of
  permutations of $X$. 
\end{example}

\begin{example} Let $X = \{x,y,z\}$. To the ideal
$I = (x, y^2, yz, z^3)$ is associated the total order $x > y > z$.
The symmetry group of this ideal only consists of the identity map.
\end{example}

\subsection{Families of ideals and algebraic sets, continued}
On the ideal $I = (x^2,y^2)$ 
perform the coordinate change $y \pil y + \frac{1}{t}x$.
The gives now the ideal $I_t$ generated by:
\begin{equation} \label{eq:pol-xyt} x^2, \quad (y + \frac{1}{t}x)^2 = \frac{1}{t^2} x^2 + 2 \frac{1}{t} xy +
  y^2. \end{equation}
Multiplying with $t^2$ and let $t \pil 0$ the limit of both polynomials
is $x^2$. This is not good geometrically. The zero set is now blown up
to a line.
There is a good limit ideal $I_0$ with the same ``size'' as the $I_t$'s,
but to see its generators we
need to throw extra generators into \eqref{eq:pol-xyt}.

Note that the generators also give the following combination:
\[ (y- 2\frac{1}{t}x) ( \frac{1}{t^2} x^2 + 2 \frac{1}{t} xy + y^2)
  + (2\frac{x}{t^3} + 3\frac{y}{t^2}) x^2 = y^3. \]
The ideal generated by \eqref{eq:pol-xyt}, when $t \neq 0$
is the same as the ideal generated by the three polynomials:
\[ x^2, xy + \frac{t}{2} y^2, y^3. \]
Now let $t \pil 0$. We get the generators:
\begin{equation} \label{eq:pol-xy0}  x^2, xy, y^3. \end{equation}
This {\it is} a generating set for the correct limit ideal $I_0$.
This ideal is also
{\it more} degenerate than the ideal $I$ generated by $x^2, y^2$:
One can {\it not} perform coordinate changes on $I_0$ and get the ideal
$I$ as a limit ideal.

\medskip
\noindent{\it Question.} Can one further degenerate the ideal $I_0$
generated by \eqref{eq:pol-xy0}?

\medskip
The answer turns out to be no. The ideal $I_0$ cannot be further degenerated.
Why is this so, and which ideals have this property?

\medskip

\subsection{Parameter spaces: The Hilbert schemes}
First let us give some theoretical insight into what is going on.
There is a good notion of parameter spaces for graded ideals in a
polynomial ring $k[x_1, \ldots, x_n]$.
These are the {\it Hilbert schemes} $H$, which are themselves algebraic sets. 
Each point $p \in H$ corresponds to a graded ideal $I_p$ in
$k[x_1, \ldots, x_n]$. 

The general linear group $\gln$ acts by coordinate changes on
ideals. For an ideal $I$ in $k[x_1, \ldots, x_n]$, a $g \in \gln$
gives by coordinate change a new ideal $g.I$.
Acting on $I_p$, the ideal $g.I_p$ will be a new graded ideal $I_{q}$
parametrized by the point $q$. We can then define an action of $\gln$
on $H$ by $g.p = q$.

Algebraic sets are endowed with a topology, the {\it Zariski topology},
and so also the Hilbert scheme $H$. For a subset $Z$ of $H$, denote by
$\overline{Z}$ its closure.
If the orbit $\cO(p)$ by $\gln$ is not closed, its boundary
$\overline{\cO(p)} \backslash \cO(p)$ contains a point $p^\prime$ and
the ideal $I_{p^\prime}$ corresponds to a more degenerate ideal than $I_p$.
The {\it most degenerate ideals} $I_p$ are then those ideals such that the
orbit $\cO(p)$ is closed.
As any ideal can be degenerated to a monomial
ideal, any such closed orbit contains a monomial ideal.
The monomial ideals in such orbits
can be described precisely.



\subsection{Most degenerate ideals and Borel-fixed ideals}
An ideal which is as degenerate as possible, is a monomial
ideal with a ``direction'':

\begin{theorem} \label{thm:pol-totpre} (char.$k = 0$) A monomial
  ideal is a most degenerate ideal
  (i.e. it cannot be degenerated further, its orbit on the Hilbert scheme
  is closed), iff the associated preorder by \eqref{eq:pol-montopre}
  is a {\bf total preorder}.
\end{theorem}

This formulation is new. It follows from Theorem \ref{thm:pol-borelclosed}
below, and we now give the traditional presentation of this.

\medskip
Put a total order $x_1 > x_2 > \cdots > x_n$ on the variables.
By Section \ref{sec:ass}, the associated
subgroup $B(n) = \glxo{}$ of $\gln$
is the subgroup of upper triangular matrices.

\begin{definition} \label{def:pol-borel}
  A graded ideal $I$ in $k[x_1, \ldots, x_n]$
  is {\it Borel-fixed} if $g.I = I$ for all $g$ in the Borel group
  of upper triangular matrices in $\gln$.
\end{definition}

\begin{definition} \label{def:pol-stst}
  A monomial ideal $I$ in $k[x_1, \ldots, x_n]$ is
  {\it strongly stable} if whenever $m$ is a monomial in $I$ and
  $m = x_j n$ and $x_i > x_j$ (i.e. $i < j$), then $x_in \in I$.
\end{definition}

As it turns out, in characteristic zero Definition \ref{def:pol-stst}
is a combinatorial
characterization of Definition \ref{def:pol-borel}:

\begin{fact} When $k$ has characteristic zero, a graded ideal $I$ is
  Borel-fixed iff it is monomial and strongly stable, \cite[Thm. 15.23]{Eis}
  or \cite[Prop. 4.24]{HeHi}.
\end{fact}

\begin{example}
  The ideal $J = (x^2, xy, y^3)$ is strongly stable for the
  monomial order $x > y$.
\end{example}

We see that a monomial ideal is strongly stable iff the associated
preorder is a total preorder in $\Pre(X)$ which is $\bsucceq$ the total order
\[ x_1 > x_2 > \cdots > x_n \]
in $\Pre(X)$. 
Note that the associated preorder of a monomial ideal is a total preorder
iff it is given by permuting the variables of a strongly stable ideal.




The following is shown in \cite{FShift}, but is really an old
and known result which follows from standard theory, Borel's fixed point
theorem \cite[Thm. 21.2]{Hum}. It is equivalent to Theorem \ref{thm:pol-totpre}.

\begin{theorem} \label{thm:pol-borelclosed}
  (char.$k = 0$) An orbit of $\gln$ on
  the Hilbert scheme $H$ is closed iff it is the orbit of a Borel-fixed
  ideal.
\end{theorem}

Borel-fixed ideals or strongly stable ideals are standard topics in
textbooks on computational or combinatorial commutative algebra,
\cite{Eis, HeHi, MS}. They
are essential for establishing results on the numerical behavior of
homogeneous ideals $I$ in polynomial rings, for instance the classical
results of Macaulay on what are the possible
Hilbert functions $h(n) = \dim_k I_n$ (dimension of the $n$'th graded piece)
of such ideals. They are essential in {\it shifting theory}
\cite[Section 11]{HeHi} where they are used to understand numerical behavior of
combinatorial objects, see \cite{KaBj} for a strong classification
result.

\subsection{Monomial orders}
For any homogeneous ideal in a polynomial ring there is a standard
way of getting directly to a degenerate ideal, by introducing a {\it monomial order} : This is a total order on the monomials in a polynomial ring with the 
following two simple extra requirements:
\begin{itemize}
\item $x_i > 1$ for each variable $x_i$,
\item If $m,n$ are monomials with $m > n$, then $x_im > x_in$ for each variable $x_i$.
\end{itemize}

With a mononomial order specified,
one gets the theory of Gröbner bases for ideal in polynomial rings,
\cite[Section 15]{Eis}, \cite{CoCoa}.
This is the central tool of computational algebra. We must however
emphasize that there are many monomial orders to choose from (a priori
none is canoncial), and
different orders have different advantages for computational tasks.
The setting of a monomial order is thus quite orthogonal to the
philosophy of this article. We aim at uncovering {\it naturally} occurring order
structures, while a monomial order must be {\it chosen}.

A monomial order gives for any ideal $I$ a distinguished direct
monomial ideal degeneration
(one arrives at it by a single one-parameter $t$ degeneration),
the {\it initial ideal} $\text{in}(I)$.
By performing a slight perturbation of
$I$ we have the following classical result by Andr\'e Galligo from
1974 \cite{Gal}, \cite[Thm. 15.20]{Eis}, which brings one
directly to a maximally degenerate ideal. 

\begin{theorem} Choose any monomial order on a polynomial ring
  $k[x_1, \ldots, x_n]$ 
  such that $x_1 > x_2 > \cdots > x_n$.
  Let $I$ be a homogenous ideal in this ring and
  perform a general coordinate change on $I$
  to get $g.I$. The initial ideal $\text{in}(g.I)$ is Borel-fixed.
  This is called the  {\it generic initial ideal} of $I$ with respect
  to the monomial order.
\end{theorem}

That one can thus directly arrive at a Borel-fixed ideal (and possibly
many by varying the monomial order) has given the idea that these
could give an approach to classify the components of the parameter
spaces, Hilbert schemes, for homogeneous ideals \cite{FlRo, Lella, RR, RS}.
A classification appears very intricate however, and must
be said to be at a rudimentary level.





\subsection{Strongly stable ideals as up-sets}
Monomials in $k[x_1, \ldots, x_n]$ are in bijection with $\NN^n$
where $\NN = \{ 0, 1, 2, \ldots \}$, or equivalently with
$\Hom(\overline{n}, \NN)$ where $\overline{n}$ is the anti-chain
$\overline{n} = \{ 1, 2, \ldots, n \}$. For $\bfu = (u_1,u_2,
\ldots, u_n)$ the corresponding monomial is $x^\bfu = x_1^{u_1}x_2^{u_2}
\ldots x_n^{u_n}$. Monomial ideals in $k[x_1, \ldots, x_n]$ are in
bijection with up-sets $U \sus \Hom(\overline{n}, \NN)$, with the monomial
ideal consisting of all $x^{\bfu}$ with $\bfu \in U$.

\medskip
Let $[n] = \{ 1 < 2 < \ldots < n \}$. The monomials in $k[x_1, \ldots, x_n]$
are also in bijection with order-preserving maps $\Hom([n],\NN)$.
For $u_1 \leq u_2 \leq \cdots \leq u_n$ in the latter, the corresponding
monomial is $x_1^{u_1}x_2^{u_2 - u_1} \cdots x_n^{u_n - u_{n-1}}$. As it turns out
strongly stable ideals are precisely
in bijection with up-sets $U \sus \Hom([n],\NN)$,
\cite{FPP, FShift} or first observed in \cite{FGH}.
Moreover, given a total preorder $T$, the strongly stable
ideals whose associated preorder is $\succeq \, T$, correspond
precisely to the up-sets in $\Hom(T^\op, \NN)$.

\section{Bipartite graphs}
\label{sec:bip}

To combinatorial objects one may often associate algebraic objects.
In particular, simplicial complexes have an associated Stanley-Reisner
ring, \cite[Section 1]{HeHi}. So, one may consider algebraic conditions on the ring and investigate
what this means for the combinatorial object.
For commutative rings a basic niceness condition is that of being
{\it Cohen-Macaulay}. It has various equivalent definitions, often adapted
to various specializations.
We give one of them which does not require much technical background, and
investigate this in a basic combinatorial setting.

\subsection{Graph ideals}
Let $G = (E,X)$ be a {\it simple} graph with edges $E$ and vertices $X$.
Let $k[X]$ be the polynomial ring whose variables are the $x \in X$. 
To each edge $\{v,w\} \in E$ there is a quadratic monomial $vw \in k[X]$.
Let $I(G)$ be the monomial ideal generated by all such $vw$ coming from
edges in $G$, the {\it edge ideal} of $G$, and let $k[G] = k[X]/I(G)$
be the quotient algebra, the {\it edge ring} of $G$.

\begin{example} \label{eks:bip-abcd} The graph
\[ \linegraph \]
gives the edge ideal
\[ I_1 = (ab,bc,cd) \sus k[a,b,c,d]\]
and the edge ring $A_1 = k[a,b,c,d]/(ab,bc,cd)$.

We may associate geometry to this ideal, by taking the common zeros
$Z(ab,bc,cd)$ of the monomials in $I_1$. This is an algebraic set
living in ${\mathbb A}^4_k$, the affine four-space.
Denoting such points by
$\bfu = (u_a,u_b,u_c,u_d)$, note that if $\bfu$ is in the zero set
then also $t \bfu$ is in the zero set, so we may consider the zero set
to live in the projective three-space ${\mathbb P}^3$.
It is easy to see that the zero set is then the union of the three
projective lines
\[ Z(a,c) \cup Z(c,b) \cup Z(b,d) \]
where for instance $Z(a,c)$ is the projective line where the
coordinates $a = 0$ and $c = 0$, so it consists of points
$(0,t,0,s)$ ($s,t$ not both zero) in ${\mathbb P}^3$.
\end{example}

\begin{example}
  The graph
  \[ \squaregraph \]
gives the edge ideal $I_2 = (ab,bc,cd,da)$ and
edge ring $A_2 = k[a,b,c,d]/I_2$.
And the zero set of $I_2$ is the union of two disjoint lines in
${\mathbb P}^3$:
\[ Z(a,c) \cup Z(b,d). \]
\end{example}

Since the generators of the ideals $I_1$ and $I_2$ are homogeneous polynomials,
these are graded ideals, and the quotient rings are graded rings
$A_1$ and $A_2$ 
(more specifically graded $k$-algebras).

\subsection{Some notions from commutative algebra}
Let us recall a few notions from commutative algebra. 
A ring $A$ is {\it graded} if it can be written
\[ A = A_0 \oplus A_1 \oplus \cdots \oplus A_n \oplus \cdots  \]
where the $A_i$ are abelian groups, and the products $A_i \cdot A_j
\sus A_{i+j}$. If $A$ is a $k$-algebra, the $A_i$ are $k$-vector spaces.

An element $a$ in a commutative ring is a {\it nonzerodivisor} if
$a b = 0$ implies $b = 0$. A sequence of elements of $A$:
$a_1, a_2, \ldots, a_n$ is {\it regular} if each $a_i$ is a nonzerodivisor
in the quotient ring $A/(a_1, \ldots, a_{i-1})$ and these quotient rings
are all non-zero, for $i = 1, \ldots, n+1$.

A graded $k$-algebra $A$ is {\it Cohen-Macaulay} if there is a regular
sequence $a_1, a_2, \ldots, a_n$ such that $A/(a_1,a_2, \ldots, a_n)$
is finite-dimensional as a vector space over $k$. This is a basic {\it niceness}
condition on graded $k$-algebras. The Krull dimension of $A$ is then said
to be $n$, the length of the regular sequence. 

\begin{example} For the ring $A_1$ above, $a-b, c-d$ is a regular
sequence. The quotient ring
\[ A_1/(a-b,c-d)  \iso k[a,c]/(a^2, ac, c^2) \iso k \oplus
\langle a, c \rangle \]
is a three-dimensional vector space.
\end{example}

\begin{example} For the ring $A_2$, for the sequence $a-b,c-d$, the
first $a-b$ is a nonzerodivisor in $A_2$, but $c-d$ is a zerodivisor
in
\[ A_2/(a-b) \iso k[a,c,d]/(a^2, ac, ad, cd) \]
as $(c-d)a = 0$ but $a \neq 0$ in $A_2/(a-b)$.

One may show there is no regular sequence in $A_2$ giving a quotient
ring which is a finite-dimensional vector space, so $A_2$ is not
Cohen-Macaulay.
\end{example}

Geometrically the reason why $A_2$ is not Cohen-Macaulay is that the
zero set of the ideal is two non-intersecting lines. In general, if
the zero set is an algebraic subset which is not "sufficiently connected",
this implies it is not Cohen-Macaulay.

\medskip
Edge ideals of graphs $I(G)$ were first introduced in 1990 by Raphael Villarreal
\cite{Vill}. Since then, it has been an active topic in combinatorial
commutative algebra to understand algebraic properties of edge rings
and connect it to combinatorial properties of graphs \cite{MoVill},
to build a dictionary between commutative algebra and graph theory.
For instance,
minimal vertex covers of a graph corresponds to minimal associated
prime ideals of the edge ideal $I(G)$. 
Somewhat surprisingly, for a graph $G$ in general,
it seems as difficult question as anything
to find a general criterion on $G$ for the edge ring $k[G]$ to be
Cohen-Macaulay.
But, for one class of graphs we have a good answer: Bipartite graphs.

\subsection{When are bipartite graphs Cohen-Macaulay?}
A simple graph is {\it bipartite} if there is a disjoint union
$X = A \sqcup B$ such that every edge in $G$ has one vertex in $A$
and the other in $B$. An edge $\{a,b\}$ (with $a \in A, b \in B$)
may be identified with the ordered pair $(a,b)$.
Hence a bipartite graph is equivalent to a subset $R \sus A \times B$,
that is, a {\it relation} between $A$ and $B$.
We now characterize when $k[G]$ is Cohen-Macaulay for a bipartite graph.

Let $(V,\geq)$ be a partially ordered set.
The partial order is relation
$R \sus V\times V$, and we write $v \leq w$ if $(v,w) \in R$.
For $i = 1,2$ write $V_i = V \times \{i\}$.
It induces a bipartite
graph with vertices $X = V_1 \sqcup V_2$, and edges $\{ (v,1),(w,2) \}$ when
$v \leq w$. Let $L$ be the associated edge ideal.

\begin{theorem} \label{thm:bip-VCM}
  The edge ring $k[V_1 \sqcup V_2]/L$ is a Cohen-Macaulay
  ring. Its Krull dimension equals the cardinality of $V$.
\end{theorem}

The proof of this is rather straightforward. It is not hard to show
that the sequence of variable differences $(v,1) - (v,2)$ for $v \in V$
forms a regular sequence. The final quotient ring is isomorphic to
\[ k[V]/(vw \, | \, v \leq w). \]
This quotient ring is finite-dimensional as a vector space over $k$, 
as the quotienting ideal contains the second power
$v^2$ of each variable $v \in V$.
The amazing thing is that {\it every} Cohen-Macaulay edge ideal is
of this kind: The edge ring is Cohen-Macaulay precisely when the relation between
$A$ and $B$ is a partial order \cite{HeHi-bip}, \cite[Sec.9.1.3]{HeHi}

\begin{theorem} \label{thm:bip-poset}
  Let $G$ be a bipartite graph with vertex partition $A \sqcup B$.
  The edge ring $k[G]$ is a Cohen-Macaulay ring iff there is a bijection
  $A \iso B$ such that the bipartite graph, considered as a relation
  between $A$ and $B$, may be identified with a {\bf partial
  order} on $A$. 
\end{theorem}

\begin{example}
  Let $A = \{a,c\}$ and $B =\{b,d\}$. Let $V = \{v,w\}$ with the order relation $v \leq w$.
  We identify $B$ with $V_1$ and $A$ with $V_2$. 
  The edge ring in Example \ref{eks:bip-abcd}
  becomes the edge ring
  \[ k[V_1 \sqcup V_2]/((v,1)(v,2), (v,1)(w,2), (w,1)(w,2)) \]
  which is Cohen-Macaulay by Theorem \ref{thm:bip-VCM} above.
\end{example}

\begin{remark} A related niceness condition on monomial ideals is that
  of having a {\it linear resolution}. When does this hold for edge
  ideals of bipartite graphs? This is easier
  than the theorem above. This happens when $A$ and $B$ can be given total
  orders such that the relation $R \sus A \times B$ is a up-set for
  the partial order on $A \times B$.
\end{remark}

\subsection{Letterplace ideals} Theorems \ref{thm:bip-VCM} and
\ref{thm:bip-poset} have generated much activity in
combinatorial commutative algebra. 
We briefly sketch one direction having led to considerable generalizations,
which the author has been active in.
For a set $R$ let $k[R]$ be the polynomial ring whose variables are the
elements of $R$. Each subset $S \sus R$ gives a monomial
$m_S = \prod_{s \in S} s$.

For two preorders $P$ and $Q$, we have the polynomial ring $k[P \times Q]$.
Let $\phi : P \pil Q$ be an order preserving map and let $\Hom(P,Q)$ be the set
of such. The graph of $\phi$ is $\Gamma \phi = \{ (p,\phi(p))\, |\, p \in P\}$,
which is a subset of $P \times Q$. It gives a monomial $m_{\Gamma \phi}$
in $k[P \times Q]$. These monomials generated an ideal denoted $L(P,Q)$.

When $P = [n]$ and $Q = V$, a partially ordered set,
$L([n],V)$ is called a {\it letterplace ideal} \cite{FGH}.
These are Cohen-Macaulay ideals \cite{EHM}, \cite{FGH},
and are generated by monomials of degree $n$.

\begin{example} The ideals $L([2],V)$ are precisely the ideals given
  in Theorem \ref{thm:bip-VCM}.
\end{example}

A more general framework is given in \cite{FPP, FGH}. Consider
a {\it finite} down-set $\cI$ in $\Hom(V, \NN)$. The ideal $L(V;\cI)$ generated
by all $m_{\Gamma \phi}$ for $\phi \in \cI$ is called a {\it co-letterplace
  ideal}. Its Alexander dual monomial ideal $L(\cI;V)$ is a
(generalized) {\it letterplace ideal}, a class considerably generalizing the
$L([n],V)$. These ideals are also Cohen-Macaulay, and the simplicial complex
associated via the Stanley-Reisner correspondence is always a
{\it triangulated ball} \cite{DFN}.
In the recent \cite{FM} this led to a fundamental theorem on the
combinatorial geometry of monomial ideals: Any polarization of
an artin monomial ideal gives, via the Stanley-Reisner correspondence,
a simplicial complex which is a triangulated ball.

\appendix

\section{Functors and adjunctions}

In line with the general philosophy in mathematics, there should
be structure preserving maps between objects, and so also between
categories.

\subsection{Functors}

\begin{definition} Let $\cC, \cD$ be categories.
A {\it functor} $F : \cC \pil \cD$:
\begin{itemize}
\item Associates to each object $c$ in $\cC$ an object $F(c)$ in $\cD$.
\item To each pair of objects $c_1,c_2$ of $\cC$ there is a function
\[ \Hom(c_1,c_2) \mto{F} \Hom(F(c_1), F(c_2)) \] such that for morphisms
$c_1 \mto{f} c_2 \mto{g} c_3$ we have
$F(g \circ f) = F(g) \circ F(f)$.
\item $F(\ben_C) = \ben_{F(c)}$. 
\end{itemize}
\end{definition}

\begin{example}
If $\cC$ is a category with one object $c$ and similarly $\cD$ with
one object $d$, a functor $F : \cC \pil \cD$ is a map of
monoids $F: \Hom(c,c) \pil \Hom(d,d)$. If $\cC$ and $\cD$ are groups,
this is a group homomorphism.
\end{example}

\begin{example} If $P$ and $Q$ are preorders, then considering them
as categories, a functor $F : P \pil Q$ is simply an order preserving map
from $P$ to $Q$, i.e. if $a \leq b$ in $P$, then  $F(a) \leq F(b)$ in $Q$.
\end{example}

\subsection{Adjunctions}

A map of sets $f : X \pil Y$ has an inverse iff it is a bijection.
Often this will not be so, but there may be maps $g  : Y \pil X$ which
are approximations to an inverse. Furthermore, the approximation may be
from "above" or from "below".

Let $P,Q$ be preorders and consider an order-preserving map $g : P \pil Q$.
There may then be a map $f : Q \pil P$ which approximates 
an inverse of $g$ from "above". Precisely this is formulated as follows.
For $q \in Q$ and $p \in P$ we have:
\[ f(q) \leq p \text{ iff } q \leq g(p). \]
So given $g$ and $q \in Q$ we consider all $p$ with $q\leq g(p)$
and make $f(q)$ as large as possible $\leq $ these $p$.

Similarly, we may consider $h : Q \pil P$ which may be an approximation
from "below". Then
\[ g(p) \leq q \text{ iff } p \leq h(q). \]

\begin{example}

  Let $P$ and $Q$ be the natural numbers $\NN = \{0, 1, 2, \cdots \}$.
  The map $f : \NN \pil \NN$ given by $f(n) = 2n$ is an order preserving
  map. If one wants an approximate ``inverse'' to this,
  one should of course send
  $6$ to $3$ and $8$ to $4$. However, for $7$ it is a question whether one
  should choose $3$ of $4$. There are two consistent choices:
  \begin{itemize}
  \item $g^l : \NN \pil \NN$ given by $g^l(m) = \lceil \frac{m}{2} \rceil$,
  \item $g^r : \NN \pil \NN$ given by $g^r(m) = \lfloor \frac{m}{2} \rfloor$,
\end{itemize}
These are both order preserving maps, and are respectively the left
and the right adjoint to the map $f$.
\end{example}

The general categorical notion of this is the following.

\begin{definition} Let $F : \cC \pil \cD$ and $G : \cD \pil \cC$ be
functors. These are {\it adjoint functors}, with $F$ left adjoint and $G$
right adjoint, if there for each $c$ in $\cC$ and $d$ in $\cD$ are
bijections
\[ \Hom_{\cD}(F(c),d) \iso \Hom_{\cC}(c, G(d)). \]
Moreover, these bijections must be natural in the sense that if $c^\prime \pil c$
and $d \pil d^\prime$ are morphisms, the diagrams derived from this
are commutative diagrams.
\end{definition}


\bibliographystyle{amsplain}
\bibliography{biblio}

\providecommand{\bysame}{\leavevmode\hbox to3em{\hrulefill}\thinspace}
\providecommand{\MR}{\relax\ifhmode\unskip\space\fi MR }
\providecommand{\MRhref}[2]{%
  \href{http://www.ams.org/mathscinet-getitem?mr=#1}{#2}
}
\providecommand{\href}[2]{#2}
\begin{thebibliography}{10}

\bibitem{ASpace}
Pavel Alexandroff, \emph{Diskrete r\"aume}, Rec. Math.[Mat. Sbornik] NS
  \textbf{2} (1937), no.~44, 3.

\bibitem{ASS}
Ibrahim Assem, Daniel Simson, and Andrzej Skowronski, \emph{{Elements of the
  Representation Theory of Associative Algebras: Volume 1: Techniques of
  Representation Theory}}, Cambridge University Press, 2006.

\bibitem{ARS}
Maurice Auslander, Idun Reiten, and Sverre~O Smal{\o}, \emph{{Representation
  theory of Artin algebras}}, no.~36, Cambridge university press, 1997.

\bibitem{BiLa}
Garrett Birkhoff, \emph{Lattice theory}, vol.~25, American Mathematical Soc.,
  1940.

\bibitem{KaBj}
Anders Bj{\"o}rner and Gil Kalai, \emph{{An extended Euler-Poincar{\'e}
  theorem}}, Acta Mathematica \textbf{161} (1988), no.~1, 279--303.

\bibitem{Bor}
ZI~Borevich, \emph{A description of the subgroups of the complete linear group
  that contain the group of diagonal matrices}, Journal of Soviet Mathematics
  \textbf{17} (1981), 1718--1730.

\bibitem{DaPr}
Brian~A Davey and Hilary~A Priestley, \emph{Introduction to lattices and
  order}, Cambridge university press, 2002.

\bibitem{DFN}
Alessio D’Al{\`\i}, Gunnar Fl{\o}ystad, and Amin Nematbakhsh,
  \emph{{Resolutions of co-letterplace ideals and generalizations of Bier
  spheres}}, Transactions of the American Mathematical Society \textbf{371}
  (2019), no.~12, 8733--8753.

\bibitem{EiML}
Samuel Eilenberg and Saunders MacLane, \emph{General theory of natural
  equivalences}, Transactions of the American Mathematical Society \textbf{58}
  (1945), 231--294.

\bibitem{Eis}
David Eisenbud, \emph{Commutative algebra: with a view toward algebraic
  geometry}, vol. 150, Springer Science \& Business Media, 2013.

\bibitem{EHM}
Viviana Ene, J{\"u}rgen Herzog, and Fatemeh Mohammadi, \emph{{Monomial ideals
  and toric rings of Hibi type arising from a finite poset}}, European Journal
  of Combinatorics \textbf{32} (2011), no.~3, 404--421.

\bibitem{FPP}
Gunnar Fl{\o}ystad, \emph{{Poset ideals of P-partitions and generalized
  letterplace and determinantal ideals}}, Acta Mathematica Vietnamica
  \textbf{44} (2019), no.~1, 213--241.

\bibitem{FShift}
\bysame, \emph{Shift modules, strongly stable ideals, and their dualities},
  Transactions of the American Mathematical Society, Series B \textbf{10}
  (2023), no.~21, 670--714.

\bibitem{FGH}
Gunnar Fl{\o}ystad, Bj{\o}rn~M{\o}ller Greve, and J{\"u}rgen Herzog,
  \emph{Letterplace and co-letterplace ideals of posets}, Journal of Pure and
  Applied Algebra \textbf{221} (2017), no.~5, 1218--1241.

\bibitem{FM}
Gunnar Fl{\o}ystad and Amir Mafi, \emph{Polarizations of artin monomial ideals
  define triangulated balls}, arXiv preprint arXiv:2212.09528 (2022).

\bibitem{FlRo}
Gunnar Fl{\o}ystad and Margherita Roggero, \emph{{Borel degenerations of
  arithmetically Cohen--Macaulay curves in ${\mathbb P}^3$}}, International
  Journal of Algebra and Computation \textbf{24} (2014), no.~05, 715--739.

\bibitem{ACT}
Brendan Fong and David~I Spivak, \emph{An invitation to applied category
  theory: seven sketches in compositionality}, Cambridge University Press,
  2019.

\bibitem{FMS}
Christopher~A Francisco, Jeffrey Mermin, and Jay Schweig, \emph{Borel
  generators}, Journal of Algebra \textbf{332} (2011), no.~1, 522--542.

\bibitem{Gal}
Andr{\'e} Galligo, \emph{{A propos du th{\'e}oreme de pr{\'e}paration de
  Weierstrass}}, Fonctions de plusieurs variables complexes, Springer, 1974,
  pp.~543--579.

\bibitem{HHfinite}
Randall~D Helmstutler and Ryan~S Higginbottom, \emph{Finite topological spaces
  as a pedagogical tool}, PRIMUS \textbf{22} (2012), no.~1, 64--74.

\bibitem{HeHi-bip}
J{\"u}rgen Herzog and Takayuki Hibi, \emph{{Distributive lattices, bipartite
  graphs and Alexander duality}}, Journal of Algebraic Combinatorics
  \textbf{22} (2005), 289--302.

\bibitem{HeHi}
\bysame, \emph{Monomial ideals}, Springer, 2011.

\bibitem{Hum}
James~E Humphreys, \emph{Linear algebraic groups}, vol.~21, Springer Science \&
  Business Media, 2012.

\bibitem{StSp}
Peter~T Johnstone, \emph{Stone spaces}, vol.~3, Cambridge university press,
  1982.

\bibitem{CoCoa}
Martin Kreuzer and Lorenzo Robbiano, \emph{Computational commutative algebra},
  vol.~1, Springer, 2000.

\bibitem{Lam}
Tsit-Yuen Lam, \emph{Lectures on modules and rings}, vol. 189, Springer Science
  \& Business Media, 2012.

\bibitem{Lella}
Paolo Lella, \emph{{An efficient implementation of the algorithm computing the
  Borel-fixed points of a Hilbert scheme}}, Proceedings of the 37th
  International Symposium on Symbolic and Algebraic Computation, 2012,
  pp.~242--248.

\bibitem{LV}
Jean-Louis Loday and Bruno Vallette, \emph{Algebraic operads}, vol. 346,
  Springer, 2012.

\bibitem{May-FC}
J~Peter May, \emph{Finite spaces and larger contexts}, Draft, book sumbitted to
  AMS (2016).

\bibitem{McCo}
Michael~C McCord, \emph{Singular homology groups and homotopy groups of finite
  topological spaces}, Duke Math. J. \textbf{33} (1966), no.~1, 465--474.

\bibitem{MS}
Ezra Miller and Bernd Sturmfels, \emph{Combinatorial commutative algebra}, vol.
  227, Springer Science \& Business Media, 2005.

\bibitem{MoVill}
Susan Morey and Rafael~H Villarreal, \emph{Edge ideals: algebraic and
  combinatorial properties}, Progress in commutative algebra \textbf{1} (2012),
  85--126.

\bibitem{Mo58}
Kiiti Morita, \emph{Duality for modules and its applications to the theory of
  rings with minimum condition}, Science Reports of the Tokyo Kyoiku Daigaku,
  Section A \textbf{6} (1958), no.~150, 83--142.

\bibitem{RR}
Ritvik Ramkumar, \emph{{Hilbert schemes with two Borel-fixed points}}, Journal
  of Algebra \textbf{617} (2023), 17--47.

\bibitem{RS}
Alyson Reeves and Mike Stillman, \emph{Smoothness of the lexicographic point},
  Journal of Algebraic Geometry \textbf{6} (1997), no.~4, 235.

\bibitem{Sko}
Andrzej Skowro{\'n}ski and Kunio Yamagata, \emph{Frobenius algebras}, vol.~12,
  European Mathematical Society, 2011.

\bibitem{Steiner}
Anne~K Steiner, \emph{The lattice of topologies: structure and
  complementation}, Transactions of the American Mathematical Society
  \textbf{122} (1966), no.~2, 379--398.

\bibitem{Vill}
Rafael~H Villarreal, \emph{{Cohen-Macaulay graphs}}, {Manuscripta Mathematica}
  \textbf{66} (1990), no.~1, 277--293.

\end{thebibliography}

\end{document}